\documentclass[12pt]{article}
\usepackage{arydshln}
\usepackage{amsmath}
\usepackage{multirow}
\usepackage{amssymb}
\usepackage{amsfonts}
\usepackage{setspace}
\usepackage{hyperref}
\usepackage{cleveref}
\usepackage{enumitem}
\usepackage{enumerate}
\usepackage{amsthm}
\usepackage{epsf}
\usepackage{graphicx}
\usepackage{changepage}
\usepackage{cases}
\usepackage[justification=centering]{caption}
\usepackage{cite}
\usepackage{xcolor}
\usepackage{amsmath}
\usepackage{amssymb}

\allowdisplaybreaks
\def\mineappendix{
        \setcounter{section}{1}
        \setcounter{subsection}{1}
        \def\thesection{\Alph{section}}
        \def\sectionap{\@startsection  {section}{1}{\z@}
                        {-3.5ex plus-1ex minus-.2ex} {0ex plus.2ex}
                        {\reset@font\Large\bf  Appendix:  \, }
                        }
        }
\makeatother
\def\Proclaim #1. #2\par{\bigbreak\noindent{\sc#1.\enspace}{\it#2}\par}

\font\Bbbfont=msbm10
\newfam\msbfam
\textfont\msbfam=\Bbbfont \textfont\msbfam=\Bbbfont

\def\pa{\partial}

\def\wt{\widetilde}

\def\wt{\widetilde}

\def\nn{\nonumber}

\def\nd{\noindent}

\newtheorem{Theorem}{Theorem}
\newtheorem{Lemma}{Lemma}

\title{Vortex Filaments in Hermitian Reductive Lie Algebras}

\author{Qing Ding \footnote{Email:qding@fudan.edu.cn}\\
	\it Department of Mathematics, Wenzhou University \\
\it Wenzhou 325035, P.R. China\\
and School of Mathematical Sciences\\
Fudan University, Shanghai 200433, P.R. China\\
Xiayu Dong \footnote{Email:xydong1996@163.com}\\
	\it School of Mathematical Sciences  \\
	\it Xiamen University, Xiamen 361005, P.R. China\\
Shiping Zhong \footnote{Email: zhongshiping@gnnu.edu.cn}\\
	\it School of Mathematics and Computer Sciences,  \\
\it Gannan Normal University, Ganzhou 341000, P.R. China}
	\date{}

\usepackage{float}
\begin{document}

\maketitle

\begin{abstract}
It is well-known that the investigation of vortex filaments (i.e., moving curves) in the Euclidean 3-space $\mathbb R^3$ is an attractive topic both in physics and mathematics. The theory consists mainly of the three vortex models, up to the third-order approximation. Such a theory has been successfully extended to Hermitian symmetric Lie algebras in mathematics with physical and geometrical backgrounds. This article is devoted to developing it to Hermitian reductive Lie algebras in a purely geometric way. The three vortex models obtained in this article fulfill that when the Hermitian reductive Lie algebra ${\mathfrak g}$ equi-collapses to a Hermitian symmetric Lie algebra ${\mathfrak h}$, they revert respectively to those in ${\mathfrak h}$.

\end{abstract}

\nd Keywords: Moving curve, Homogeneous space, Prescribed curvature representation

\nd AMS Classification: 76B47, 53A04, 53C30, 14M15, 17B80, 35Q51

\section *{\S 1. Introduction}
Vortex filaments in fluid physics, or moving
curves in a Riemannian or pseudo-Riemannian manifold in mathematics, are an attractive topic both in physics and mathematics.
The study can be traced back to
Helmholtz in 1858 (see \cite{Hel, Hel*}) and Kelvin in 1867 (see
\cite{Ke}). Since then, the vortex filament dynamics has been studied extensively
(see \cite{AH,Da,Hama,Lamb,FM,FMo,MGK}), as one
need to characterize the motion behavior of vortex filaments
appearing in various physical models, such as kinematics of
interfaces in crystal growth \cite{Lan,BK}, viscous fingering in a
Hele-Shaw cell \cite{ST}, charged fluid in a neutralizing background
\cite{UIY}, superconductors, superfluid \cite{BJOS}, etc.

The systematic investigation of vortex filaments in the Euclidean 3-space $\mathbb R^3$ was initialized by Arms and Hama (see \cite{AH}) in 1965 via so-called
``localized induction equation (LIE)''. The model they obtained is
also called the binormal motion of space curves in $\mathbb R^3$, first derived by Luigi Sante Da Rios
in 1906 (\cite{Da}). This development has effectively stimulated mathematical methods and techniques to reveal the analytic and geometric properties, hidden structures and symmetries of vortex filaments. The first three (physical) corrected
vortex models in $\mathbb
R^3$ are the Da Rios's model: ${\gamma}_t=\gamma_{x}\times \gamma_{xx}$ (\cite{Da,AH}), the Fukumoto-Miyazaki's model: ${\gamma}_t={\gamma}_{xxx}+\frac{3}{2}{\gamma}_{xx}\times({
\gamma}_x\times {\gamma}_{xx})$ (\cite{FM}) and the Fukumoto-Maffatt's model: ${\gamma}_t=\bigg\{\nu[\kappa^2\tau{\bf
T}+(2\kappa_x\tau+\kappa\tau_x){\bf
N}+(\kappa\tau^2-\kappa_{xx}){\bf B}]+\mu\kappa^3{\bf B}\bigg\}$ (\cite{FMo}), in which $\{{\bf T},{\bf N},{\bf B}\}$ is the Frenet
frame, $\kappa$ and $\tau$ are respectively the curvature and torsion along the filament curve ($\nu$ and $\mu$ are free real parameters).
The physical significance of the above leading-order, second-order and third-order vortex models was given in \cite{Da,AH,FM,FMo}. The study of dynamical and geometric properties of these three models constitute the theory of vortex filaments in $\mathbb R^3$, up to the third-order approximation, in mathematics.
The leading-order and the second-order models are completely integrable systems, but the third-order one is not the case except for certain choices of parameters $\mu$ and $\nu$ (e.g., $\mu=-\frac{1}{2}\nu$).
This gap phenomenon shows a mysterious and interesting aspect of the dynamics of moving curves in $\mathbb R^3$. In order to exploit geometric properties of these vortex models, the geometric concepts of Schr\"odinger flow, geometric KdV flow and generalized bi-Schr\"odinger flow
are proposed respevtively in \cite{DW}, \cite{SW} and \cite{DW2} (refer to \S2, for a brief introduction of them). A detailed description about analytic and geometric properties of them in $\mathbb R^3$ is stated and summarized in Introduction in \cite{DW2}. Below, we make a table to summarize it.
\begin{table}[!h]
\centering Known vortex models in $\mathbb R^3$ up to the third-order approximation

and their geometric interpretations

\footnotesize
\begin{tabular}{|c|c|c|c|}
	\hline
	 & Vortex models in $\mathbb R^3$ & Geometric interpretation& By  Hasimoto transform\\
	\hline
Leading-order & Da Rios's model &  Schr\"{o}dinger flow from & NLS:\\
 vortex model & $\gamma_t=\gamma_x\times\gamma_{xxx}$ (integrable)& $\mathbb R $ to $\mathbb S^2$ & ~$\varphi_t=\varphi_{xx}+2|\varphi|^2\varphi$ \\
	\hline
2nd-order & Fukumoto-Miyazaki's & Geometric KdV flow from & cKdV:   \\
vortex model & model (integrable) & $\mathbb R $ to $\mathbb S^2$  &  $\varphi_t=\varphi_{xxx}
+6|\varphi|^2\varphi_x$ \\
	\hline
3rd-order &Fukumoto-Moffatt's& Generalized bi-Schr\"{o}dinger  & A fourth-order  \\
 vortex model & model (non-integrable)  & flow from $\mathbb R $ to $\mathbb S^2$  & NLS-like equation \\
    \hline
\end{tabular}
\end{table}

\medskip
On the other hand, in the past decades, much effort has been
devoted to extending the theory of ``localized induction hierarchy'' in the 3-dimensional Euclidean space $\mathbb R^3$ to higher
dimensional spaces. A successful generalization in literature is the ``localized induction matrix hierarchy'' associated to a Hermitian symmetric
Lie algebra (see, for
example, \cite{DW2,FK,LP,OS,TeUh}), with
physical and geometrical backgrounds. Fordy and Kulish in
\cite{FK} constructed the matrix nonlinear Schr\"odinger equation associated to the Hermitian symmetric Lie algebra ${\mathfrak h}=\mathfrak u(n)$ ($n\ge2$). They showed that the matrix nonlinear Schr\"odinger
equation is in fact equivalent to a generalized
Heisenberg ferromagnet model. Langer and Perline applied the technique of
Sym (\cite{Sym}) and Pohlmeyer (\cite{Poh}) to produce a geometric
realization of the matrix nonlinear Schr\"odinger equation in ${\mathfrak h}=\mathfrak u(n)$ (\cite{LP}). In
this process, they illuminated the ``localized
induction matrix hierarchy'' for arclength-parameterized curves
evolving in ${\mathfrak h}=\mathfrak u(n)$. The first
three of the hierarchy are:
\begin{eqnarray}
\gamma_t&=&\gamma_x,\,\nonumber\\
\gamma_t&=&{-}[\gamma_x,\,\gamma_{xx}],\,\label{GLIE}\\
\gamma_t&=&{-(}\gamma_{xxx}+\frac{3}{2}[\gamma_{xx},\,
[\gamma_x,\,\gamma_{xx}]]{)},\label{GKDV}
\end{eqnarray}
where $[\cdot,\cdot]$ denotes the Lie bracket in ${\mathfrak h}$. The models (\ref{GLIE}) and (\ref{GKDV}) are exactly the leading-order (i.e., generalized
Heisenberg ferromagnet) and the second-order vortex
models in ${\mathfrak h}$, since when ${\mathfrak h}=\mathfrak{su}(2)$, Eqs.(\ref{GLIE}) and (\ref{GKDV}) revert to the Da Rios's and Fukumoto-Miyazaki's models, respectively.
In order to find the third-order vortex model in ${\mathfrak h}$, one has to overcome the difficulty that arises from the non-integrability. In 2018 the first author and Wang derived in \cite{DW2} the following model:
\begin{eqnarray}
\gamma_t&=&\alpha\bigg([\gamma_x,\,\gamma_{xxxx}]
-[\gamma_{xx},\,\gamma_{xxx}]\bigg)+(4\beta-2\alpha)
\left[\gamma_{x},\,
\gamma_{xx}\gamma_x^{-1}\gamma_{xx}
\gamma_x^{-1}\gamma_{xx}\right],\label{GFM1}
\end{eqnarray}
with $\gamma_x\in\{E^{-1}\sigma_3E | E\in U(n)\}\hookrightarrow \mathfrak u(n)$,
in ${\mathfrak h}=\mathfrak u(n)$ by using the Hamiltonian gradient flow of a fourth-order functional which is called the generalized bi-Schr\"odinger flow, where $\alpha, \beta$ are real parameters and $\sigma_3$ is given by (\ref{04}) below.
It is shown in \cite{DW2} that when ${\mathfrak h}=\mathfrak{su}(2)$, Eq.(\ref{GFM1}) reverts to the Fukumoto-Moffatt's model in $\mathbb R^3$
(when $\alpha=-\nu$ and $\beta=-\frac{1}{4}\mu$). Therefore, three basic vortex models in $\mathfrak h=\mathfrak u(n)$ are established that constitute the theory of vortex filaments in Hermitian symmetric Lie algebras, up to the third-order approximation. Below, we also make a table to present these models and together with their geometric interpretations.

\begin{table}[!h]
\centering Vortex models in the Hermitian symmetric Lie algebra ${\mathfrak h}=\mathfrak u(n)$ up to the

third-order approximation and their geometric interpretations
\footnotesize
\begin{tabular}{|c|c|c|c|}
	\hline
	$$& Vortex in a symmetric Lie  & Geometric & By gauge \\
	$$& algebra $\mathfrak u(u)$ & interpretation &  transformation \\
	\hline
Leading-order & Generalized Heisenberg ferro- & Schr\"{o}dinger flow from $\mathbb R $ to & Matrix NLS \\
 vortex model & magnet model (\ref{GLIE}) (integrable)& Grassmannian &\\
	\hline
2nd-order & Langer and Perline's & Geometric KdV flow from $\mathbb R $ to&  Matrix cKdV  \\
vortex model & model (\ref{GKDV}) (integrable) &  Grassmannian  &  \\
	\hline
3rd-order & Ding and Wang's & Generalized bi-Schr\"odinger flow  &4th-order matrix \\
	vortex model & model (\ref{GFM1}) (non-integrable)  & from $\mathbb R^1$ to Grassmannian  &  NLS-like equation\\
    \hline
\end{tabular}
\end{table}

The aim of this paper is to develop the three basic vortex models and to characterize their dynamical properties, up to the third-order approximation, in a Hermitian reductive Lie algebra ${\mathfrak g}=\mathfrak u(n)$ in a purely geometric way. The exploration is based on the works of Fordy and Kulish (\cite{FK}) in deducing the generalized nonlinear Schr\"odinger equation (i.e., the $N$-wave equation called in \cite{FK}), Langer and Perline (\cite{LP}) in constructing the geometric recursion operator in a Hermitian symmetric Lie algebra and the first author and his collaborators (\cite{dingzhu,DW}) in introducing the geometric concept of PDEs with prescribed curvature representations. We point out that the constructions of the leading-order and the second-order vortex models in the Hermitian reductive Lie algebra ${\mathfrak g}=\mathfrak u(n)$ can still be handled along the line of integrable systems. However, the method cannot be applied to obtain the third-order vortex model since it is non-integrable in general. Meanwhile, the geometric method via generalized bi-Schr\"odinger flows applied in \cite{DW2}
cannot be used directly too.
So, our strategy is to generalize the geometric recursion operator to Hermitian reductive Lie algebras and use it to derive the third-order vortex model with the aid of PDEs with prescribed curvature representation in the category of Yang-Mills theory.
It is well-known that, associated with a Hermitian reductive non-symmetric Lie algebra, there is a corresponding homogeneous complex flag manifold. Therefore, undoubtedly, some geometric and algebraic properties of complex flag manifolds are applied.
As a by-product, the Fordy-Kulish's NLS and the Langer-Perline's LIE hierarchies are extended successfully to a Hermitian reductive Lie algebra. Their interrelationship is well determined.
It looks very natural that the three basic vortex models we founded in a Hermitian reductive Lie algebra are somewhat more complicated than those in a Hermitian symmetric Lie algebra. Despite of this, these models also have nice analytic and geometric properties. When the Hermitian reductive Lie algebra ${\mathfrak g}=\mathfrak u(n)$ equi-collapses to a Hermitian symmetric Lie algebra ${\mathfrak h}=\mathfrak u(n)$, the three basic vortex models in ${\mathfrak g}$ revert exactly and respectively to the models (\ref{GLIE}), (\ref{GKDV}) and (\ref{GFM1}) in ${\mathfrak h}$. This relates to the fascinating mechanism hidden in the reduction of the complex flag manifold ${\rm F}_s^{\rm \mu}(k_1,\cdots,k_s)$ to the symmetric space $G(k,n-k)$. We believe that the exploitation of vortex filaments in Hermitian reductive Lie algebras is significant and applicable in the theory of moving curves in geometry and Lie algebras.

The paper is organized as follows. In section 2, we briefly review some basic facts concerning Lie algebras, homogeneous spaces, the geometric recursion operator in a Hermitian symmetric Lie algebra and PDEs with prescribed curvature representation. In \S 3, based on the extension of the geometric recursion operator to the Hermitian reductive Lie algebras and PDEs with zero or prescribed curvature representation, we derive three basic vortex models in a Hermitian reductive non-symmetric Lie algebra $\mathfrak g=\mathfrak u(n)$. In \S4, we show that in the Hermitian reductive Lie algebra $\mathfrak g=u(n)$, the $(L+2)$-th Fordy-Kulish's NLS equation is equivalent to the $L$-th Langer-Perline's LIE equation, that extends greatly the famous Da Rios-NLS correspondence. In \S5, we prove that when the Hermitian reductive Lie algebra ${\mathfrak g}$ equi-collapses to a Hermitian symmetric Lie algebra ${\mathfrak h}$, the obtained three basic models revert exactly and respectively to those in ${\mathfrak h}$.

\section *{\S2. Preliminaries}

In this section, we review briefly some basic facts concerning Lie algebras, symmetric or homogeneous spaces, and the geometric recursion operator in a Hermitian symmetric Lie algebra. The geometric concepts of Schr\"odinger flow (\cite{DW}), geometric KdV flow (\cite{SW}),
generalized bi-Schr\"odinger flow (\cite{DW2}) and PDEs with zero or prescribed curvature representations (\cite{dingzhu,DW2}) are also reviewed concisely.

In terms of the Cartan-Weyl basis, a complex simple Lie algebra ${\mathfrak g}$ has the following commutation relations (refer to \cite{FK,He}, for example)

(i) $[h_i,h_j]=0,~~\forall h_i,h_j\in {\mathfrak k}$,

(ii) $[h,e_{\alpha}]=\alpha(h)e_{\alpha},~~\forall h\in {\mathfrak k},~\alpha\in\Phi$,

(iii) $[e_{\gamma},e_{-\gamma}]=h_{\gamma}=\sum\limits^l_{i=1}d_ih_i$,

(iv) $[e_{\gamma},e_{\zeta}]=\left\{\begin{array}{c}N_{\gamma,\zeta}~e_{\gamma+\zeta},~~~~0\not=\gamma+\zeta\in \Phi\\
0,~~~~~~~~~~~~~~~~~~~\gamma+\zeta\notin \Phi.\end{array}\right.$

\noindent Here
${\mathfrak k}$ is the Cartan subalgebra, which is the maximal abelian subalgebra of diagonalizable elements in ${\mathfrak g}$.
${\mathfrak k}$ has basis $\{h_i\}^l_{i=1}$ and $d_i$ are the components of $[e_{\gamma},e_{-\gamma}]\in {\mathfrak k}$ with respect to this basis. The number $l$ is the rank of ${\mathfrak g}$.
$\alpha: {\mathfrak k}\to \mathbb C$ are linear functionals, called roots, on ${\mathfrak k}$ and their values on given $h\in {\mathfrak k}$ are eigenvalues of the matrix $ad_h$. The corresponding eigenvectors $e_{\alpha}$ are called root vectors. $\Phi$ is the set of roots.
The coefficients $N_{\alpha,\zeta}$ are the most complicated part of these commutation relations.

A homogeneous space is a differential manifold $M$ endowed with a transitive action of a Lie group ${\bf G}$. The subgroup of ${\bf G}$ which leaves a given point $o\in M$ fixed is called the isotropy group at $o$ and is expressed as
\begin{eqnarray}
{\bf K}={\bf K}_o=\{g\in G~|~g\cdot o=o\}.\nonumber\end{eqnarray}
Such a $M$ is in fact diffeomorphic to the coset space ${\bf G/K}= \{gK: g \in G\}$.
Let $\mathfrak {g}$ and $\mathfrak {k}$ be respectively the Lie algebras of ${\bf G}$ and ${\bf K}$, and ${\mathfrak m}$ be the vector space complement of ${\mathfrak k}$ in ${\mathfrak g}$. Then we have
\begin{eqnarray}
{\mathfrak g}={\mathfrak k}\oplus {\mathfrak m},~~[{\mathfrak k},{\mathfrak k}]\subset {\mathfrak k}.\nonumber\end{eqnarray}
At this moment, we only have
$[{\mathfrak k},{\mathfrak k}]\subset {\mathfrak k}$, but no information about $[{\mathfrak k},{\mathfrak m}]$ and $[{\mathfrak m},{\mathfrak m}]$.

When ${\mathfrak g}={\mathfrak k}\oplus {\mathfrak m}$ fulfils the conditions:
\begin{eqnarray}\label{01}
[{\mathfrak k},{\mathfrak k}]\subset {\mathfrak k},~~[{\mathfrak k},{\mathfrak m}]\subset {\mathfrak m},\end{eqnarray}
then $M$, or equivalently $G/K$, is called a {reductive} homogeneous space and, correspondingly, ${\mathfrak g}$ is called a reductive Lie algebra.
In this case, it is well-known that the subspace ${\mathfrak m}$ is identified with the tangent space $T_oM$ of $M={\bf G/K}$ at the point $o$ and
$G/K$ possesses canonically defined connections, but not every one of them is Levi-Civita.
Moreover, when the reductive Lie algebra ${\mathfrak g}={\mathfrak k}\oplus {\mathfrak m}$ satisfies the further condition:
\begin{eqnarray}\label{02}
~~[{\mathfrak m},{\mathfrak m}]\subset {\mathfrak k},\end{eqnarray}
then $M$, or equivalently $G/K$, is called a symmetric homogeneous space and hence ${\mathfrak g}$ is called a symmetric Lie algebra.
In this case, the decomposition of ${\mathfrak g}={\mathfrak k}\oplus {\mathfrak m}$ is orthogonal with respect to the Killing metric on ${\mathfrak g}$ and $G/K$ carries a canonical Levi-Civita connection induced by the Killing metric.
We would point out that, among reductive homogeneous spaces with the group action being the adjoint action, there are two classes: one is the class of symmetric spaces and another one is the class of flag manifolds which are not symmetric.

Let's consider the unitary group $G=U(n)$ of degree $n$. Its Lie algebra is $\mathfrak{u}(n) = \{ X \mid X = -X^*,~ X \in gl(n,\mathbb C) \}$, in which $gl(n,\mathbb C)$ stands for the set of all $n\times n$ complex matrix and $*$ denotes the operator of the complex conjugate transpose. Let $\omega$ be a diagonal matrix in $\mathfrak{u}(n)$ given by
\begin{eqnarray}\omega = \frac{\sqrt{-1}}{2}
\begin{pmatrix}
\mu_1 I_{k_1} &0 & &&0 \\
0& \mu_2 I_{k_2} & & &0 \\
& & & \ddots & \\
0& 0& & & \mu_s I_{k_s}
\end{pmatrix},\label{03}\end{eqnarray}
where $2\le s\le n$, $I_k$ stands for the $k\times k$ unit matrix, $k_i\in\mathbb{N}_+ \text{~with~}\sum\limits_{i=1}^s k_i=n$ and $\mu_i\in \mathbb R~(i=1,2,\cdots,s)$ with $\mu_i \neq \mu_j, \forall~ 1\leq i,j\leq s, ~i\neq j$. We denote by $\mathbb C_{k_i\times k_j}$ the set of all the complex $(k_i\times k_j)$-matrices. The Lie algebra $\mathfrak{u}(n)$ of $U(n)$ is now decomposed with respect to $\omega$ by
\begin{eqnarray}
\mathfrak{u}(n)={\mathfrak k}\oplus {\mathfrak m},\label{decomp}
\end{eqnarray}
where
$$
{\mathfrak k}=\hbox{\rm
Kernel}(ad_{\omega})=\left\{\left. \begin{pmatrix}
A_1 &0 & & &0 \\
0& A_2 & & &0 \\
& & & \ddots & \\
0& 0& & & A_s
\end{pmatrix}\in \mathfrak{u}(n)~\right|~A_i\in \mathfrak u(k_i)\right\}$$
and
$${\mathfrak m}=\left\{\left.\left(\begin{array}{cccc}0&\varphi_{12}&\cdots&\varphi_{1s}\\
-{\varphi}_{12}^*&0&\cdots&\varphi_{2s}\\
\vdots&\vdots&\ddots&\vdots\\
-{\varphi}_{1s}^*&-{\varphi}_{2s}^*&\cdots&0\end{array}\right)~\right|~\hbox{block~matrices}~\varphi_{ij}\in \mathbb {C}_{k_i\times k_j}, 1\leq i<j\leq s\right\}.
$$
It is obvious that the above $\mathfrak k$ and $\mathfrak m$ satisfy the conditions of a reductive Lie algebra, i.e., $[{\mathfrak k},{\mathfrak k}]\subset {\mathfrak k},~~[{\mathfrak k},{\mathfrak m}]\subset {\mathfrak m}$. This is one of the standard Hermitian reductive Lie algebras of Type A III.
For the space ${\rm F}_s^{\rm \mu}(k_1,k_2,\cdots,k_s) = \{ E^{-1} \omega E | E\in U(n) \}$, where ${\rm \mu}=(\mu_1,\cdots,\mu_s)$, we see that the group $U(n)$ acts transitively on it by the adjoint action. The isotropy subgroup at the point $\omega$ is $K_\omega = \{E\in U(n)|E^{-1}\omega E=\omega\}\cong U(k_1)\times U(k_2)\times \cdots \times U(k_s)$. Hence, ${\rm F}_s^{\rm \mu}(k_1,k_2,\cdots,k_s)$ is actually diffeomorphic to the complex homogeneous coset space $U(n)/U(k_1)\times U(k_2)\times \cdots \times U(k_s)$. Especially, ${\rm F}_n^{\rm \mu}(1,1,\cdots,1)=U(n)/U(1)\times U(1)\times \cdots \times U(1)$ is the full complex flag manifold. When $3\le s\le n-1$, ${\rm F}_s^{\rm \mu}(k_1,k_2,\cdots,k_s)$ is a generalized complex flag manifold.
When $s = 2$, $k_1=k, k_2=n-k$ for some integer $1\le k<n$ and $\mu_1 = 1 ,\mu_2 =-1$, we replace $\omega$ in this case by
\begin{eqnarray}\label{04}\sigma_3=\frac{\sqrt{-1}}{2}\left(\begin{array}{cc}
I_k&0\\0&-I_{n-k}\end{array}\right).\end{eqnarray}
Then ${\rm F}_2^{(1,-1)}(k,n-k)= U(n)/U(k)\times U(n-k)=\{ E^{-1} \sigma_3 E | E\in U(n) \}$ is actually a complex symmetric space, which is exactly the complex compact Grassmannian manifold $Gr(k,n-k)$. It is well-known that $Gr(k,n-k)$ is a K\"ahler manifold which has $J=ad_{\sigma_3}$ as its compatible complex structure.

Next, we turn to the geometric recursion operator on a Hermitian symmetric Lie algebra introduced by Langer and Perline in \cite{LP}. It is well-known that
a $1+1$-dimensional nonlinear integrable equation arises as the integrability condition for an overdetermined
linear system:
\begin{eqnarray}
\phi_x=U(x,t;\lambda)\phi,\quad \phi_t=V(x,t;\lambda)\phi, \label{05}
\end{eqnarray}
where $x$ is the position variable, $t$ is the time variable and $\lambda$ is a free spectral parameter. The eigenfunction $\phi=\phi(x,t;\lambda)$ takes its values in a Lie group $G$, and meanwhile, the coefficients $U(x,t;\lambda)$ and $V(x,t;\lambda)$ take their values in the Lie algebra $\mathfrak g$ of $G$. Associated with the Hermitian symmetric Lie algebra $\mathfrak h=\mathfrak u(n)={\mathfrak k}\oplus{\mathfrak m}$ given by (\ref{decomp}) with $s = 2$, $k_1=k, k_2=n-k$ for some integer $1\le k<n$ and $\mu_1 = 1 ,\mu_2 =-1$, Fordy and Kulish in \cite{FK} set
$U(x,t;\lambda)=\lambda\sigma_3+Q(x,t)$ with the potential function $Q=Q(x,t)\in\mathfrak m$ and $V(x,t;\lambda)=V_L:=\sum\limits_{j=0}^{L}V^{(j)}(x,t)\lambda^{L-j}$ following so-called the {\it polynomial ansatz} for a given positive integer $L\ge2$, in which the coefficient $V^{(j)}(x,t)$ ($0\le j\le L$) is a function of $Q$ and its derivatives independent of $\lambda$. Cross-differentiating the equations in (\ref{05}) gives the integability condition:
$
U_t-(V_L)_{x} + [U,V_L] = 0.
$
From the integrability condition, one may obtain $V_L$ explicitly and hence the $(L+1)$-th Fordy-Kulish's NLS equation. For example, in the case of $L=2$, one finds that $V_2=\lambda^2 \sigma_3+\lambda Q+\frac{1}{2}[Q,[Q,\sigma_3]]+[Q_x,\sigma_3]$ and hence the matrix nonlinear Schr\"odinger equation of focusing type: $\sqrt{-1}Q_t+Q_{xx}+2QQ^*Q=0$.

With the Lax pair (\ref{05}) given above, the derivation of $V=V_L$ can be deduced and summarized as a recursion scheme (refer to \cite{LP}). In fact, by introducing a recursion operator $\widetilde{\mathcal{R}}_{sym}$
which takes $\mathfrak m$-fields $B$ to $\mathfrak m$-fields,
\begin{eqnarray}\label{RT-symmetric}
\widetilde{\mathcal{R}}_{sym}(B)=-(\partial_x-ad_Q\partial^{-1}_xad_Q)JB,
\end{eqnarray}
where $\partial_x^{-1}$ stands for the antiderivative with respect to $x$ and $J=ad_{\sigma_3}$,
we set ${\widetilde X}^{(1)}=JQ$, ${\widetilde X}^{(2)}=\widetilde{\mathcal{R}}_{sym}({\widetilde X}^{(1)})=Q_x$, $\cdots$, ${\widetilde X}^{(L+1)}=\widetilde{\mathcal{R}}_{sym}({\widetilde X}^{(L)})$ and etc. Then the evolution equation $Q_t={\widetilde X}^{(L)}$ gives the $L$-th Fordy-Kulish NLS equation (refer to \cite{LP}) and all these equations constitute the Fordy-Kulish's NLS hierarchy in $\mathfrak h$. It should be mentioned that in this process, the 3rd Fordy-Kulish's NLS hierarchy equation is exactly the matrix nonlinear Schr\"odinger equation (refer to \cite{FK,LP}).

The followings are some facts about Sym-Pohlmeyer curves in the Hermitian symmetric Lie algebra $\mathfrak h=\mathfrak u(n)$ and the related Langer-Perline's LIE hierarchy. Using the Lax pair (\ref{05}) in $\mathfrak h$ indicated above, we set $\{B\}=\phi^{-1}B\phi$ when $B\in \mathfrak h$ and $\{B,C\}=\{[B,C]\}$ for $B,C\in \mathfrak h$, where $\phi=\phi(x,t;\lambda)\in U(n)$ is the eigenfunction of (\ref{05}). Then, we introduce $W(x,t,\lambda)=\phi_\lambda\phi^{-1}$ in this case, and consider the $t$-family curves in {$\mathfrak h$}:
$$
\gamma(x,t;\lambda)=\{W\}=\phi^{-1}\phi_\lambda.$$
It is a direct verification that $\gamma_x=\{W\}_x
=\{\sigma_3\}$. Denoted by $\langle\cdot,\cdot\rangle=-\hbox{tr} (XY)$ the Cartan-Killing metric on $\mathfrak h$, we see that by Ad-invariance of $\langle\cdot,\cdot\rangle$, $\langle\gamma_x,\gamma_x\rangle=\langle\sigma_3,\sigma_3\rangle=l$, where $l=\hbox{dim}(\mathfrak m)$. Hence, $\gamma$ will be a family of arc-length parameterized evolving curves in $\mathfrak h$ with respect to the rescaled Ad-invariant metric $\frac{1}{l}\langle\cdot,\cdot\rangle$ on $\mathfrak h$. This is referred as a family of arclength-parameterized Sym-Pohlmeyer curves, or simply a Sym-Pohlmeyer curve, in $\mathfrak h$.

In the Hermitian symmetric Lie algebra $\mathfrak h=\mathfrak u(n)=\mathfrak k\oplus \mathfrak m$,  the integrability condition of (\ref{05}) and together with the symmetric conditions (\ref{01},\ref{02}) yield $(V_{\mathfrak k})_x=[Q,V_{\mathfrak m}]$. Hence the $\mathfrak k$-component of $V$ is determined by the $\mathfrak m$-component of $V$. This leads Langer and Perline to introduce in \cite{LP} an operator $\mathcal{K}$, which takes $\mathfrak m$-fields to $\mathfrak k$-fields:
    $$
       \mathcal{K}(B_{\mathfrak m})=\partial_x^{-1}[Q,B_{\mathfrak m}],~~ B\in \mathfrak h.
    $$
Thus a Sym-Pohlmeyer vector field is defined of the form $Y=\{\mathcal{K}(B_{\mathfrak m})+B_{\mathfrak m}\}$ for some $\mathfrak m$-field $B_{\mathfrak m}$ (up to integration constant in $\mathcal{K}$). There are two operators introduced in \cite{LP} on vector fields $Y=\{B\}$ along $\gamma$ as follows.

(i) Renormalization operator:
$$
    \mathcal{P}(\{B\})=\{\mathcal{K}(B_{\mathfrak m})+B_{\mathfrak m}\}=\{\partial_x^{-1}[Q,B_{\mathfrak m}]+B_{\mathfrak m}\}.
$$

(ii) Geometric recursion operator (in the symmetric case, it is denoted by $\mathcal{R}_{sym}$):
$$
    \mathcal{R}_{sym}(Y)=-\mathcal{P}([T,\partial_x Y]),
$$
where $T$ is the unit tangent vector field along $\gamma$ (given by $T=\gamma_x=\{\sigma_3\}$).

For a Sym-Pohlmeyer curve $\gamma$ with potential function $Q$ indicated above and $\lambda=0$, let vector fields $X^{(L)}$ be defined along $\gamma$ by the induction: $X^{(0)}=\{\sigma_3\}=T$, $X^{(1)}={\mathcal R_{sym}}(X^{(0)})=\{Q\}=-[\gamma_x,\gamma_{xx}]$, $\cdots$, $X^{(L)}={\mathcal R_{sym}}(X^{(L-1)})$ ($L\ge1$). The curve evolution equation
$\gamma_t=X^{(L)}$ is called the $L$-th Langer-Perline's LIE equation in the Hermitian symmetric Lie algebra $\mathfrak h=\mathfrak u(n)$. All of them constitute the Langer-Perline's LIE hierarchy in $\mathfrak h$ (refer to \cite{LP}).
It is proved in \cite{LP} that there is a nice correspondence between the $L$-th Langer-Perline's LIE equation $\gamma_t=X^{(L)}$ and the $(L+2)$-th Fordy-Kulish's NLS equation $Q_t={\widetilde X}^{(L+2)}$ for arbitrary integer $L\ge1$, which generalizes the famous Da Rios-NLS correspondence saying that if a curve $\gamma$ in $\mathbb R^3$ evolves according to the Da Rios equation, then the associated complex function $\varphi(x,t) = \kappa(x,t) \exp{\left(\sqrt{-1}\int^x
\tau(s,t)ds\right)}$, called the Hasimoto transform  (see \cite{Ha}), evolves according to the focusing nonlinear Schr\"odinger equation, where $\kappa$ and $\tau$ stand for the curvature and torsion of $\gamma$ at $(x,t)$, respectively.

We end this section with a brief review of the geometric concepts of generalized bi-Schr\"odinger flows, Schr\"odinger flows, KdV flows and PDEs with prescribed curvature representations. It is well-known that for a smooth map $u$ from a Riemannian manifold $(M,g)$ to another one $(N,h)$, the energy and the bi-energy functionals are defined respectively by $E(u) = \int_M|du|^2dv_g$ and $E_2(u) = \int_M|(d+d^*)^2u|^2dv_g$, where $d^*$ denotes the co-differential operator of $d$. When the target manifold $N$ is K\"ahler that admits a compatible complex structure $J$, we define a fourth-order generalized bi-energy functional of a map $u$ from $(M,g)$ to $(N,h,J)$ by
$$
E_{\eta,\alpha,\beta}=\eta E(u)+\alpha E_2(u)+\beta\int_M \langle R(\nabla u,J_u\nabla u)J_u\nabla u,\nabla u\rangle dv_g,
$$
where $\eta$, $\alpha$ and $\beta$ are real parameters and $R$ is the curvature operator of $N$. A map $u = u(x,t): M\times [0, T) \to  N$, where $T>0$ is a real number, is
called a generalized bi-Schr\"odinger flow from $M$ to $N$ if $u$ fulfills the equation of
the Hamiltonian gradient flow of $E_{\eta,\alpha,\beta}(u)$ (\cite{DW2}), that is
$$u_t = J_u\nabla E_{\eta,\alpha,\beta}(u).$$
When $\eta\not=0$ and $\alpha=\beta=0$, the generalized bi-Schr\"odinger flow reduces to Schr\"odinger flow (\cite{DW}). Usually, the Schr\"odinger flow (resp. generalized bi-Schr\"odinger flow) from $\mathbb R^1$ to $N$ is called one-dimensional (1-d) Schr\"odinger flow (resp. 1-d generalized bi-Schr\"odinger flow) on $N$ in literature.
The geometric KdV flow of maps $u$ from the real line ${\mathbb R}^1$ into a K\"ahler manifold $(N,J,h)$ (see \cite{SW}) is defined as that $u$ satisfies
\begin{eqnarray}
\frac{\pa}{\pa t}u=\nabla_x^2u_x+\frac{1}{2}R(u_x,J_uu_x)J_uu_x,\label{001}
\end{eqnarray}
where $\nabla_x$ is the covariant derivative
$\nabla_{\frac{\pa}{\pa x}}$ on the pull-back bundle $u^{-1}TN$
induced from the Levi-Civita connection on $N$ and $\nabla_xu$ is denoted by $u_x$.
Finally, we come to the geometric concept of PDEs with zero or prescribed curvature representations in the category of Yang-Mills theory. One sees that the
integrable Eq.(\ref{GLIE}) is rewritten as: $\varphi_t=-[\varphi,\,\varphi_{xx}]$,
where $\varphi={\gamma}_x$. Defining a family of connections with a spectral parameter
$\lambda$  on the trivial bundle $\mathbb R^2\times U(n)$ by
$$
A=\lambda\varphi dx+\left(\lambda^2\varphi-\lambda
[\varphi,\varphi_x]\right)dt,
$$
we have that Eq.(\ref{GLIE}) is equivalent to the zero curvature representation of $A$:
$$
F_A=dA-A\wedge A=-\lambda\left(\varphi_t+[\varphi,\varphi_{xx}]\right)dx\wedge dt=0.
$$
Thus Eq.(\ref{GLIE}) is called as an equation with zero curvature representation. For a given partial differential equation (PDE), if there
exists a $\lambda$-family of connections 1-form $A$ and a $\lambda$-family 2-form
$K$ on the trivial bundle $\mathbb R^2\times U(n)$, such that the equation is equivalent to the holding of the formula:
\begin{eqnarray}
F_A=dA-A\wedge A=K, \label{bi14}
\end{eqnarray}
then it is called an equation with prescribed curvature
representation. When $K=0$, it returns to the zero curvature representation. We should point out that, roughly speaking,
a PDE with prescribed non-zero curvature representation is non-integrable in general.

Terng and Uhlenbeck proved in \cite{TeUh} that the equation of $1$-d Schr\"odinger flow on the Grassmannian manifold $Gr(k,n-k)$ ($1\le k\le n-1$) produces the leading-order model (\ref{GLIE}) in $\mathfrak h$. The first author and He demonstrated in \cite{dinghe} that the second-order vortex model (\ref{GKDV}) in $\mathfrak h$ is equivalent to the equation of geometric KdV flow from $\mathbb R^1$ to $Gr(k,n-k)$.
The first author and Wang showed in \cite{DW2} that the third-order vortex model (\ref{GFM1}) in $\mathfrak h$ is equivalent to the equation of $1$-d generalized bi-Schr\"odinger flow on $Gr(k,n-k)$. In \cite{Onodera}, Onodera generalized the study of 1-d generalized bi-Schr\"odinger flow on complex Grassmannian manifolds to a 1-d fourth-order curve flow on general K\"ahler manifolds.

\section* {\S3 Construction of models in reductive Lie algebras}
In this section, we shall apply techniques in integrable systems to derive the leading-order and the second-order vortex models in the Hermitian reductive Lie algebra $\mathfrak g=\mathfrak u(n)$. Consequently, the Fordy-Kulish's NLS and the Langer-Perline's LIE integrable hierarchies in $\mathfrak g$ are obtained. The non-integrable third-order vortex model in $\mathfrak g$ is derived with the aid of a generalized geometric recursion operator and the geometric concept of PDEs with prescribed non-zero curvature representations. These three basic models in ${\mathfrak g}$ fulfill that when ${\mathfrak g}$ equi-collapses to a Hermitian symmetric Lie algebra ${\mathfrak h}$, they revert exactly and respectively to those known models in ${\mathfrak h}$, as we will see in \S5.

We now focus on the Lax pair (\ref{05}) in the Hermitian reductive Lie algebra $\mathfrak g$. Here, the eigenfunction $\phi=\phi(x,t;\lambda)$ takes its values in the unitary Lie group $G=U(n)$ with its Lie algebra $\mathfrak g={\mathfrak u(n)}=\mathfrak m\oplus\mathfrak k$ being a Hermitian reductive Lie algebra given by (\ref{decomp}).
As done in the case of symmetric Lie algebras, we set analogously $U(x,t;\lambda)=\lambda \omega+Q(x,t)$ with $\omega$ being given by (\ref{03}) and the potential function $Q\in \mathfrak m$, and $V=V'_L=\sum\limits_{j=0}^{L}V^{'(j)}(x,t)\lambda^{L-j}\in \mathfrak g$ is a polynomial ansatz for a given positive integer $L$, in which the coefficient $V^{'(j)}(x,t)$ ($0\le j\le L$) independent of $\lambda$ is a matrix-valued function of the entries of $Q$ and their derivatives. We should mention that the discussions below may be carried out analogously to other types of Hermitian reductive Lie algebras.

By decomposing $V^{'(j)}$ as $V^{'(j)}=V_{\mathfrak m}^{'(j)}+V_{\mathfrak k}^{'(j)}$ with $V_{\mathfrak m}^{'(j)}\in \mathfrak {m}$ and  $V_{\mathfrak k}^{'(j)}\in \mathfrak{k}$, we obtain, from the integrability condition $U_t-(V'_{L})_x+[U,V'_L]=0$, the following recurrence relations:
\begin{align}
V^{'(0)} &=\omega,  \label{eq:1} \\
[\omega,V_{\mathfrak m}^{'(j)}] &= \partial_x V_{\mathfrak m}^{'(j-1)} - [Q, V_{\mathfrak k}^{'(j-1)}] - [Q,V_{\mathfrak m}^{'(j-1)}]_{\mathfrak m}, \quad j=1, \ldots, L, \label{eq:2} \\
\partial_x V_{\mathfrak k}^{'(j)} &= [Q, V_{\mathfrak m}^{'(j)}]_{\mathfrak k}, \quad j=0, \ldots, L, \label{eq:3}
\end{align}
and the corresponding integrable equation
\begin{eqnarray}\label{FKL}
Q_t = \partial_x V_{\mathfrak m}^{'(L)} - [Q, V_{\mathfrak k}^{'(L)}] - [Q,V_{\mathfrak m}^{'(L)}]_{\mathfrak m},
\end{eqnarray}
which is an differential system of the entries of $Q$. Eq.(\ref{FKL}) is called the $(L+1)$-th Fordy-Kulish's NLS equation in $\mathfrak g=\mathfrak u(n)$.
The recursion scheme (\ref{eq:1}-\ref{eq:3}) and the dynamical integrable equation (\ref{FKL}) can be compactly described by introducing a generalized recursion operator ${\widetilde {\mathcal{R}}}_{red}:{\mathfrak m}\rightarrow {\mathfrak m}$ as follows.
\begin{eqnarray}\label{RT-reductive}
\widetilde{X}\mapsto{\widetilde {\mathcal{R}}}_{red}{\widetilde X}=(\partial_x-ad_{Q}\partial_x^{-1}ad_{Q}|_{{\mathfrak k}}-ad_{Q}|_{\mathfrak m})ad_{\omega}^{-1}{\widetilde X},~~{\widetilde X}\in {\mathfrak m},
\end{eqnarray}
where
\begin{eqnarray*}
ad_Q|_{\mathfrak k}:\mathfrak m\to \mathfrak k, X\rightarrow ad_Q|_{\mathfrak k}(X):=[Q,X]_{\mathfrak k},\\
ad_Q|_{\mathfrak m}:\mathfrak m\to \mathfrak m, X\rightarrow ad_Q|_{\mathfrak m}(X):=[Q,X]_{\mathfrak m},
\end{eqnarray*}
and the operator $ad_{\omega}^{-1}: {\mathfrak m}\rightarrow {\mathfrak m}$ is determined by
\begin{align*}
[\omega,ad_{\omega}^{-1}X]=X,~~\forall X\in \mathfrak m.
\end{align*}
By setting ${\widetilde X}^{'(1)}=ad_{\omega}Q$ and ${\widetilde X}^{'(L)}=\widetilde{\mathcal{R}}_{red}{\widetilde X}^{'(L-1)}$ for $L\ge 2$ by induction, the evolution equation $Q_t={\widetilde X}^{'(L)} (=[\omega,V^{'(L)}])$ is called the $L$-th Fordy-Kulish's NLS equation. This creates the Fordy-Kulish's NLS hierarchy ($L\ge1$) in the Hermitian reductive Lie algebra $\mathfrak g={\mathfrak u(n)}$. One notes that for given $L\ge2$, the integrability condition of the Lax pair (\ref{05}) in $\mathfrak g$ produces the $(L+1)$-th Forty-Kulish's NLS equation $Q_t={\widetilde X}^{'(L+1)}$.

We would point out that the operator $ad_{\omega}^{-1}$ has already been used in the literature (for example, see \cite{Gerdjikov2008}). This operator is particularly important and useful for us to express the dynamical models of vortex filaments in ${\mathfrak g}$, as we will see below.
For example, in the simplest Hermitian reductive non-symmetric Lie algebra $\mathfrak u(3)={\mathfrak k}\oplus{\mathfrak m}$, where
$$
{\mathfrak k}=\left\{\begin{pmatrix}
\sqrt{-1}\lambda_1  &0 & 0\\
0& \sqrt{-1}\lambda_2  & 0\\
0&0 & \sqrt{-1}\lambda_3
\end{pmatrix}~\bigg|~\lambda_1,\lambda_2,\lambda_3\in \mathbb R\right\}$$
and
$$
{\mathfrak m}=\left\{\begin{pmatrix}
0  &\varphi_1 & \varphi_3\\
-{\overline\varphi}_1& 0 & \varphi_2\\
-{\overline\varphi}_3&-{\overline\varphi}_2 & 0
\end{pmatrix}~\bigg|~\varphi_1,\varphi_2,\varphi_3\in \mathbb C\right\},
$$
we take
\begin{eqnarray*}
\omega = \frac{\sqrt{-1}}{2}
\begin{pmatrix}
\mu_1  & & \\
& \mu_2  & \\
& &  \mu_3
\end{pmatrix}\in \mathfrak k,\end{eqnarray*} in which $\mu_1,\mu_2$ and $\mu_3$ are three distinct real numbers.
The corresponding reductive homogeneous space is the full flag manifold $F_3^{\mu}(1,1,1)=U(3)/U(1)\times U(1)\times U(1)\cong \{E^{-1}\omega E~|~E\in U(3)\}$.
In this case, the operators $ad_{\omega}$ and $ad_{\omega}^{-1}$ are explicitly given by
\begin{eqnarray*}
ad_{\omega}X&=&[\omega, X]=\frac{\sqrt{-1}}{2}\left(\begin{array}{ccc}
0 & (\mu_1-\mu_2)\varphi_{1} & (\mu_1-\mu_3)\varphi_{3} \\
(\mu_1-\mu_2)\bar{\varphi}_{1} & 0 & (\mu_2-\mu_3)\varphi_{2} \\
(\mu_1-\mu_3)\bar{\varphi}_{3} & (\mu_2-\mu_3)\bar{\varphi}_{2} & 0\end{array}\right),\\
ad_{\omega}^{-1}X&=&-2\sqrt{-1}\left(\begin{array}{ccc}
0 & \frac{1}{(\mu_1-\mu_2)}\varphi_{1} & \frac{1}{(\mu_1-\mu_3)}\varphi_{3} \\
\frac{1}{(\mu_1-\mu_2)}\bar{\varphi}_{1} & 0 & \frac{1}{(\mu_2-\mu_3)}\varphi_{2} \\
\frac{1}{(\mu_1-\mu_3)}\bar{\varphi}_{3} & \frac{1}{(\mu_2-\mu_3)}\bar{\varphi}_{2} & 0\end{array}\right),
\end{eqnarray*}
where $X=\left(\begin{array}{ccc}
0 & \varphi_{1} & \varphi_{3} \\
-\bar{\varphi}_{1} & 0 & \varphi_{2} \\
-\bar{\varphi}_{3} & -\bar{\varphi}_{2} & 0\end{array}\right)\in \mathfrak m$.
When $L=2$, the $3$-rd Fordy-Kulish's NLS equation (i.e., $Q_t={\widetilde X}^{'(3)}$) in $\mathfrak u(3)$ is:
\begin{eqnarray}\label{3-wave}
\left\{\begin{array}{l}
\sqrt{-1} {\varphi_1}_{t}=2\frac{{\varphi_1}_{xx}}{\mu_1-\mu_2}+2\varphi_1\left(\frac{2\left|\varphi_1\right|^2}{\mu_1-\mu_2}-\frac{\left|\varphi_2\right|^2}{\mu_2-\mu_3}+\frac{\left|\varphi_3\right|^2}{\mu_1-\mu_3}\right)
-2\frac{{\varphi_3}_x \bar{\varphi}_2}{\mu_1-\mu_3}-2\frac{\varphi_3 \bar{\varphi}_{2x}}{\mu_2-\mu_3}, \\
\sqrt{-1}{\varphi_2}_{t}=2\frac{{\varphi_2}_{xx}}{\mu_2-\mu_3}+2\varphi_2\left(-\frac{\left|\varphi_1\right|^2}{\mu_1-\mu_2}+\frac{2\left|\varphi_2\right|^2}{\mu_2-\mu_3}+\frac{\left|\varphi_3\right|^2}{\mu_1-\mu_3}\right)
+2\frac{\varphi_3 \bar{\varphi}_{1x}}{\mu_1-\mu_2}+2\frac{{\varphi_3}_x\bar{\varphi}_1}{\mu_1-\mu_3},\\
\sqrt{-1} {\varphi_3}_{t}=2\frac{{\varphi_3}_{xx}}{\mu_1-\mu_3}+2\varphi_3\left(\frac{\left|\varphi_1\right|^2}{\mu_1-\mu_2}+\frac{\left|\varphi_2\right|^2}{\mu_2-\mu_3}+\frac{2\left|\varphi_3\right|^2}{\mu_1-\mu_3}\right)
-2\frac{{\varphi_2}_x \varphi_1}{\mu_2-\mu_3}+2\frac{\varphi_2 {\varphi_1}_x}{\mu_1-\mu_2},
\end{array}\right.
\end{eqnarray}
which is exactly Eq.(3.35), called the 3-wave equation, in \cite{FK} when we replace $\varphi_2$ (resp. $\varphi_3$) by $\varphi_3$ (resp. $\varphi_2$) and change the factor $\frac{1}{2}$ in $\omega$ by $1$. We would point out that the minus symbol in front of the term $\frac{q_{1x}q_3}{a_1-a_2}$ in the second equation of Eq.(3.35) in \cite{FK} should be plus.

With the Lax pair (\ref{05}) in the Hermitian reductive Lie algebra $\mathfrak g=\mathfrak u(n)$ indicated above, as done in the symmetric case, we introduce $W(x,t,\lambda)=\phi_\lambda\phi^{-1}$ and regard $\gamma=\{W\}=\phi^{-1}\phi_\lambda$ as a family of curves evolving in $\mathfrak g$. It is obvious that ${\gamma}_x=\{W\}_x=\{(\lambda\omega+Q)_\lambda\}=\{\omega\}$. By use of the bi-invariant Killing metric $\langle A,B\rangle =-tr(AB)$,  $\forall A,B \in \mathfrak g$, we see that $|{\gamma}_x|^2=$ constant. After properly rescaling the metric, ${\gamma}$ is a family of arclength-parameterized curves in $\mathfrak g$. Henceforth, we also refer to a $t$-family of arclength-parameterized curves in $\mathfrak g$ as a Sym-Pohlmeyer curve, especially the curve $\gamma=\phi^{-1}\phi_\lambda$ at $\lambda=0$.
The unit tangent vector of the Sym-Pohlmeyer curve $\gamma=\phi^{-1}\phi_\lambda$ is $T={\gamma}_x=\{\omega\}$.
We also introduce the operator $\mathcal{K}:\mathfrak m\rightarrow \mathfrak k: X_{\mathfrak m}\mapsto \mathcal{K}(X)=(\partial_x^{-1}[Q,X_{\mathfrak m}])_{\mathfrak k}$, $\forall X\in \mathfrak g$, as done in the case of the symmetric Lie algebra. One sees that when $\mathfrak g$ is Hermitian symmetric, the sub-index $\mathfrak k$ in $(\partial_x^{-1}[Q,X_{\mathfrak m}])_{\mathfrak k}$ disappears automatically since $[Q,X_{\mathfrak m}]\in \mathfrak k$ by the symmetric conditions. A Sym-Pohlmeyer field on $\mathfrak g$ is also defined to be of the from $Y=\{\mathcal{K}(B_{\mathfrak m})+B_{\mathfrak m}\}$ for some $\mathfrak m$-field $B_{\mathfrak m}$ (up to integration constant in $\mathcal{K}$). Let $Y=\{B\}=\phi^{-1}B\phi, B\in \mathfrak g$ be a vector field along ${\gamma}$, we define

(i) Generalized renormalization operator:

$$
\mathcal{P}(\{B\})=\left\{\mathcal{K}\left(B_{\mathfrak m}\right)+B_{\mathfrak m}\right\}=\left\{(\partial_x^{-1}\left[Q, B_{\mathfrak m}\right])_{\mathfrak {k}}+B_{\mathfrak m}\right\}.
$$

(ii) Generalized geometric recursion operator (it is now denoted by $\mathcal{R}_{red}$):

$$
\mathcal{R}_{red} (Y)=\mathcal{P}\left(ad_{\{\omega\}}^{-2}([T, \partial_x Y])\right) ,
$$
where $ad_{\{\omega\}}^{-1}$ is defined by
$$ ad_{\{\omega\}}^{-1}:\{B_{\mathfrak m}\}\mapsto\{ad_\omega^{-1}B_{\mathfrak m}\}.$$
We point out that the operator $ad_{\{\omega\}}^{-1}$ seems to be first introduced and used here. One notes that $\mathcal{R}_{red}$ returns $\mathcal{R}_{sym}$ when $\mathfrak g$ goes back to $\mathfrak h$.

For a Sym-Pohlmeyer curve ${\gamma}=\phi^{-1}\phi_\lambda$ with potential function $Q$ at $\lambda=0$ in a Hermitian reductive Lie algebra $\mathfrak g$, we similarly define a sequence of vector fields $X^{'(L)}$ ($L\ge 1$) along $\gamma$ by the generalized geometric recursion operator $\mathcal{R}_{red}$: $X^{'(L)}=\mathcal{R}_{red}X^{'(L-1)}$ from $X^{'(0)}=\{\omega\}$, $L\ge1$. The first three of them are:
\begin{eqnarray}
X^{'(0)}&=&\{\omega\},\nonumber\\
X^{'(1)}&=&\mathcal{R}_{red}X^{'(0)}=\{Q\},\label{L11}\\
X^{'(2)}&=&\mathcal{R}_{red}X^{'(1)}=\{ad_{\omega}^{-1}Q_x+\frac{1}{2}[Q,ad_{\omega}^{-1}Q]_{\mathfrak k}\},\label{L12}\\
&\cdots&\nonumber
\end{eqnarray}
One may verifies that the vector fields $X^{'(L)}$ ($L\ge 1$) can also be represented by $X^{'(L)}=\{V^{'(L)}\}$ from the recursion process, where $V^{'(L)}$ is the lowest coefficient in $\lambda$ of $V=V'_{L}$ in the Lax pair (\ref{05}) for $\mathfrak g=u(n)$.
The Langer-Perline's LIE hierarchy is now defined to be moving curves ${\gamma}$ in $\mathfrak g$ by the model ${\gamma}_t=X^{'(L)}$ ($L=0,1,2,\cdots$). We still call the model ${\gamma}_t=X^{'(L)}$ in a Hermitian reductive Lie algebra $\mathfrak g$ the $L$-th Langer-Perline's LIE equation. The explicit expressions of the first three of them are:
\begin{eqnarray}
{\gamma}_t&=&\gamma_x,\nonumber\\
{\gamma}_t
&=&-\sum\limits_{1\leq i<j \leq s}\frac{4}{(\mu_i-\mu_j)^2}([{\gamma}_x,{\gamma}_{xx}]_{ij}+[{\gamma}_x,{\gamma}_{xx}]_{ji}),\label{firstmodel} \\
{\gamma}_t
&=&-\sum\limits_{1\leq i<j \leq s}\frac{4}{(\mu_i-\mu_j)^2}\big(({\gamma}_{xxx})_{ij}+({\gamma}_{xxx})_{ji}\big)\nonumber\\
&&-\sum\limits_{1\leq i<j \leq s}\sum\limits_{1\leq m<n \leq s}\frac{16}{(\mu_i-\mu_j)^2(\mu_m-\mu_n)^2}\Big(\big[{\gamma}_{xx},[{\gamma}_x,{\gamma}_{xx}]_{ij}+[{\gamma}_x,{\gamma}_{xx}]_{ji}\big]_{mn}\nonumber\\
&&+\big[{\gamma}_{xx},[{\gamma}_x,{\gamma}_{xx}]_{ij}+[{\gamma}_x,{\gamma}_{xx}]_{ji}\big]_{nm}\Big)\nonumber\\
&&+\sum\limits_{1\leq i<j \leq s}\frac{32\sqrt{-1}}{(\mu_i-\mu_j)^5}\big[[{\gamma}_x,{\gamma}_{xx}]_{ij},[{\gamma}_x,{\gamma}_{xx}]_{ji}\big],
\label{secondmodel}
\end{eqnarray}
where $\{X\}_{ij}=(\phi^{-1}X\phi)_{ij}:=\phi^{-1}X_{ij}\phi$, in which $X_{ij}$ denotes the matrix with the entry in the $i$-row and $j$-column being that of $X$ and zero otherwise. In fact, Eqs.(\ref{firstmodel}) and (\ref{secondmodel}) are derived respectively from Eqs.(\ref{L11}) and (\ref{L12}) by applying the following formulae:
\begin{eqnarray*}
&&{\gamma}_x=\{\omega\},\quad {\gamma}_{xx}=\{\omega,Q\},\quad {\gamma}_{xxx}=\{\omega,Q_x\}+\{[\omega,Q],Q\},\\ \nonumber
&&[{\gamma}_x,{\gamma}_{xx}]=[\{\omega\},\{\omega,Q\}]=\{\omega,[\omega,Q]\}=\{ad_{\omega}^2 Q\}, \nonumber
\end{eqnarray*}
and hence
\begin{eqnarray*}
\{Q\}&=&ad_{\{\omega\}}^{-2}\{ad_{\omega}^2 Q\}=ad_{\{\omega\}}^{-2}[{\gamma}_x,{\gamma}_{xx}],\nonumber\\
&=&-\sum\limits_{1\leq i<j \leq s}\frac{4}{(\mu_i-\mu_j)^2}([{\gamma}_x,{\gamma}_{xx}]_{ij}+[{\gamma}_x,{\gamma}_{xx}]_{ji}),\nonumber\\
\{ad_{\omega}^{-1}Q_x\}&=&ad_{\{\omega\}}^{-2}\bigg({\gamma}_{xxx}+\sum\limits_{1\leq i<j \leq s}\frac{4}{(\mu_i-\mu_j)^2}[{\gamma}_{xx},[{\gamma}_x,{\gamma}_{xx}]_{ij}+[{\gamma}_x,{\gamma}_{xx}]_{ji}]\bigg)\nonumber\\
&=&-\sum\limits_{1\leq i<j \leq s}\frac{4}{(\mu_i-\mu_j)^2}\big(({\gamma}_{xxx})_{ij}+({\gamma}_{xxx})_{ji}\big)\nonumber\\
&&-\sum\limits_{1\leq i<j \leq s}\sum\limits_{1\leq m<n \leq s}\bigg(\frac{16}{(\mu_i-\mu_j)^2(\mu_m-\mu_n)^2}[{\gamma}_{xx},[{\gamma}_x,{\gamma}_{xx}]_{ij}+[{\gamma}_x,{\gamma}_{xx}]_{ji}]_{mn} \nonumber\\
&&+[{\gamma}_{xx},[{\gamma}_x,{\gamma}_{xx}]_{ij}+[{\gamma}_x,{\gamma}_{xx}]_{ji}]_{nm}\bigg), \nonumber\\
\frac{1}{2}\{[Q,ad_{\omega}^{-1}Q]_{\mathfrak
 k}\}
 &=&\sum\limits_{1\leq p<q\leq s}\frac{32\sqrt{-1}}{(\mu_p-\mu_q)^5}[[{\gamma}_x,{\gamma}_{xx}]_{pq},[{\gamma}_x,{\gamma}_{xx}]_{qp}].\\ \nonumber
 \end{eqnarray*}
Eqs.(\ref{firstmodel}) and (\ref{secondmodel}) are respectively the leading-order and the second-order vortex models in the Hermitian reductive Lie algebra $\mathfrak g=\mathfrak u(n)$, as we will see below in \S5.

What remains us to do is to pursue the third-order vortex model in $\mathfrak g=\mathfrak u(n)$. Before doing it, we shall mention that the first three vortex models in the Hermitian symmetric Lie algebra $\mathfrak h=\mathfrak u(n)$ are respectively equivalent to
\begin{eqnarray}
{\widetilde\gamma}_t&=&{-}[{\widetilde\gamma},\,{\widetilde\gamma}_{xx}],\,\label{GLIE1}\\
{\widetilde\gamma}_t&=&{-\bigg(}{\widetilde\gamma}_{xxx}+\frac{3}{2}\Big([{\widetilde\gamma}_{xx},\,
[{\widetilde\gamma},\,{\widetilde\gamma}_{x}]]+[{\widetilde \gamma}_x,[{\widetilde \gamma},{\widetilde\gamma}_{xx}]]\Big){\bigg)},\label{GKDV1}\\
{\widetilde\gamma}_t&=&[{\widetilde\gamma},\,\alpha{\widetilde\gamma}_{xxxx}+(4\beta-2\alpha)({\widetilde\gamma}_{x}{\widetilde\gamma}^{-1}{\widetilde\gamma}_{x}
{\widetilde\gamma}^{-1}{\widetilde\gamma}_{x})_x]\label{GFM11}
\end{eqnarray}
by taking the derivative with respect to $x$ on the both sides of models (\ref{GLIE}), (\ref{GKDV}) and (\ref{GFM1}), where ${\widetilde \gamma}=\gamma_x\in Gr(k,n-k)$. Hence, instead of finding the counterpart of the third-order vortex model (\ref{GFM1}) in the Hermitian reductive Lie algebra $\mathfrak g=\mathfrak u(n)$, we would like to find its equivalent counterpart of the model (\ref{GFM11}) in $\mathfrak g$. However, one notes that because of the non-integrability, the previous method from integrable theory cannot be applied. The geometric method used in \cite{DW2} seems to lose its effectiveness too, because of the difficulty that the canonical connection used by Fordy and Kulish in \cite{FK} is not torsion free.

In order to find an effective way in obtaining the third-order vortex model in $\mathfrak g$, let's recall and refine the method applied in \cite{DW2} in proving the gauge equivalence between Eq.(\ref{GFM11}) and a fourth-order nonlinear Schr\"odinger-like equation. Rewriting (\ref{GFM11}) as
\begin{equation}\label{Bi}
{\widetilde\gamma}_t=[{\widetilde\gamma},\alpha{\widetilde\gamma}_{xxxx}+4(4\beta-2\alpha)({\widetilde\gamma}_x^3)_x],
\end{equation}
where $\widetilde\gamma=\gamma_x\in G(k,n-k)=\{E^{-1}\sigma_3E~|~E\in U(n)\}$, we see that Eq.(\ref{Bi}) is actually the equation of 1-d generalized bi-Schr\"odinger flow on  $G(k,n-k)$ (see \cite{DW2}). By using the geometric concept of PDEs with prescribed curvature representations in the category of Yang-Mills theory, we shall present Eq.(\ref{Bi}) as a PDE with prescribed curvature representation. In fact, by introducing the following connection 1-form
\begin{eqnarray}
A&=&\lambda {\widetilde\gamma}
dx+\bigg\{\lambda^4\alpha{\widetilde\gamma}-\lambda^3\alpha[{\widetilde\gamma},{\widetilde\gamma}_x]-\lambda^2\alpha
\left({\widetilde\gamma}_{xx}-6{\widetilde\gamma}^2_x{\widetilde\gamma}\right)\nn\\
&&~+\lambda\left([{\widetilde\gamma},\alpha{\widetilde\gamma}_{xxx}+
4(4\beta-2\alpha){\widetilde\gamma}^3_x]-\alpha
[{\widetilde\gamma}_x,{\widetilde\gamma}_{xx}]\right)\bigg\}dt\label{bi9}
\end{eqnarray}
and the prescribed curvature 2-form
\begin{eqnarray}
K=\lambda^2\bigg(-4(4\beta+\frac{1}{2}\alpha){\widetilde\gamma}_x^3\bigg)dx\wedge dt
\label{bi13}
\end{eqnarray}
on the bundle $\mathbb R^2\times U(n)$, where $\lambda$ is a free parameter,
we see that Eq.(\ref{Bi}) is equivalent to the formula (refer to \cite{DW2})
$$
F_A=dA-A\wedge A=K.
$$
It should be mentioned that the curvature $F_A$ of $A$ used here is expressed by the formula $F_A=dA-A\wedge A$, in contrast to that used in \cite{DW2} by $F_A=dA+A\wedge A$.
It is well-known that without loss of generality, a solution to Eq.(\ref{Bi}) is presented by ${\widetilde\gamma}=E^{-1}\sigma_3E$ with $E\in U(n)$ satisfying
$E_x=PE$, $P\in {\mathfrak m}$ (see, for example, \cite{DW2}). Throughout the paper, we shall replace $Q$ by $P$ to denote the potential function in the case of Hermitian symmetric Lie algebras and keep $Q$ to denote the potential function in the case of Hermitian reductive Lie algebras. We now take the gauge matrix to be $G=E^{-1}$ and then make the gauge transformation:
\begin{eqnarray}
A\longmapsto {\widetilde A}=-G^{-1}dG+G^{-1}AG. \label{gauge1}
\end{eqnarray}
From the Yang-Mills theory, we know that the curvature $F_A$
is of homogeneous under the gauge transformation, i.e.,
\begin{eqnarray}
F_{\widetilde A}=G^{-1}F_AG=G^{-1}KG:={\wt K}.\label{gauge2}
\end{eqnarray}
By a direct calculation, we have
\begin{eqnarray*}
{\widetilde A}&=&-Ed(E^{-1})+EAE^{-1}\\
&=&(E_x dx+E_t dt)E^{-1}+E\bigg(\lambda {\widetilde\gamma}
dx+\big\{\lambda^4\alpha{\widetilde\gamma}-\lambda^3\alpha[{\widetilde\gamma},{\widetilde\gamma}_x]-\lambda^2
\left(\alpha({\widetilde\gamma}_{xx}-6{\widetilde\gamma}^2_x{\widetilde\gamma})\right)\\
&&~+\lambda\left([{\widetilde\gamma},\alpha{\widetilde\gamma}_{xxx}+
4(4\beta-2\alpha){\widetilde\gamma}^3_x]-\alpha
[{\widetilde\gamma}_x,{\widetilde\gamma}_{xx}]\right)\big\}dt\bigg)E^{-1}\\
&=&\left(\lambda
\sigma_3+P\right)dx+\bigg\{\lambda^4\alpha\sigma_3+\lambda^3\alpha
P+\lambda^2\left(2\alpha(P_x+P^2)\sigma_3\right)\nn\\
&&+\lambda \left(\alpha(-P_{xx}+2P^3-[P,P_x])-4(4\beta+\frac{1}{2}\alpha)P^3\right)+E_tE^{-1}\bigg\}dt,\label{A-conne}\\
\end{eqnarray*}
where $E_tE^{-1}=(E_tE^{-1})_{\mathfrak m}+(E_tE^{-1})_{\mathfrak k}$ with
\begin{eqnarray}
(E_tE^{-1})_{\mathfrak m}=-2\alpha
(P_{xxx}-[P,[P,P_x]])\sigma_3-32\beta
(P^3)_x\sigma_3\nonumber\label{gauge6}
\end{eqnarray}
and \begin{eqnarray}
(E_tE^{-1})_{\mathfrak k} &=&-2\alpha\left((P^2)_{xx}-3(P^2_x)\right)\sigma_3 +6\alpha
P^4\sigma_3\nn\\&&
-(8\beta+\alpha)\left(4P^4+4\int^x(PPP_sP+PP_sPP)ds\right)\sigma_3.\nonumber\label{gauge8}
\end{eqnarray}
Although the explicit expression of $\widetilde A$ is obtained, for the purpose of latter generalization, we should rewrite $\widetilde A$ by
\begin{eqnarray}\label{tA-con}
\widetilde{A}=(\lambda\sigma_3+P)dx+\hat{V}_4dt,
\end{eqnarray}
where
 \begin{eqnarray}
  \hat{V}_4 &=& \lambda^4\alpha\sigma_3+\lambda^3\alpha V^{(1)}+\lambda^2\alpha V^{(2)}+\lambda\big(\alpha V^{(3)}-(4\beta+\frac{1}{2}\alpha)[\sigma_3,ad_P^3\sigma_3]\big)\nn\\
  &&+\big(\alpha V^{(4)}+(4\beta+\frac{1}{2}\alpha)M\big),\label{V4hat}
  \end{eqnarray}
in which $V^{(i)}=V_{\mathfrak m}^{(i)}+V_{\mathfrak k}^{(i)}$ ($i=1,2,3,4$) are determined by the following iterated relations from $V^{(0)}=\sigma_3$:
\begin{flalign*}
[\sigma_3,V^{(i)}]=\widetilde{\mathcal{R}}_{sym}[\sigma_3,V^{(i-1)}]
\end{flalign*}
(that is equivalent to having $  V_{\mathfrak m}^{(i)}=-[\sigma_3,(V_{\mathfrak m}^{(i-1)})_x-[P,{V_\mathfrak k}^{(i-1)}]]$ and
$V_{\mathfrak k}^{(i)}=\partial_x^{-1}[P,V_{\mathfrak m}^{(i)}]$ )
and
\begin{flalign*}
  M=M_{\mathfrak m}+M_{\mathfrak k}~~\hbox{with}~~
  M_{\mathfrak m}=[\sigma_3,(4P^3)_x]=-[P,[P,ad_{\sigma_3}^{-1}P]_x], ~~
  M_{\mathfrak k}=\partial_x^{-1}[P,M_{\mathfrak m}].
\end{flalign*}
Moreover,
\begin{eqnarray}
{\widetilde
K}&=&G^{-1}KG=
-\lambda^2(4\beta+\frac{1}{2}\alpha)[P,[P,ad_{\sigma_3}^{-1}P]]dx\wedge dt.
\label{gauge5}
\end{eqnarray}
Thus the prescribed curvature representation formula $F_{\widetilde A}=d{\widetilde A}-{\widetilde A}\wedge {\widetilde A}={\widetilde K}$ gives that $P$ fulfills the following
fourth-order Schr\"odinger-like equation
\begin{eqnarray}
2P_t\sigma_3&=&\alpha\bigg\{P_{xxxx}-[P,[P,P_x]]_x-P\left((P^2)_{xx}-3P_x^3\right)\nn\\
&&-\left((P^2)_{xx}-3P_x^3\right)P-2(P^3)_{xx}+6P^5\bigg\}-2(8\beta+\alpha)
\bigg\{-(P^3)_{xxx}+2P^5\nn\\
&&+P\left(\int^x_0(PPP_sP+PP_sPP)ds\right)
+\left(\int^x_0(PPP_sP+PP_sPP)ds\right)P\bigg\}.\label{gauge}
\end{eqnarray}
Here, to avoid confusion, the parameter ${\gamma}$ in Eq.(49) in \cite{DW2} is changed by $\beta$ and the parameter $\beta$ in Eq.(49) in \cite{DW2} is replaced by $\alpha$. The details are referred to \S4 in \cite{DW2}.

Our approach in finding the third-order vortex model in $\mathfrak g$ is now proposed as follows. We first generalize the connection 1-form ${\widetilde A}$ and the curvature 2-form $\widetilde{K}$ associated to Eq.(\ref{gauge})
from the Hermitian symmetric Lie algebra $\mathfrak{h}=u(n)$
to the Hermitian reductive Lie algebra $\mathfrak{g}=u(n)$.
To avoid confusion again, the generalized connection 1-form and curvature 2-form are denoted by $\widetilde{A}'$ and $\widetilde{K}'$, respectively.  $\widetilde{A}'$ and $\widetilde{K}'$ should fulfill the requirement that when the Hermitian reductive Lie algebra $\mathfrak{g}$ equi-collapses (refer to \S5 for the definition) to the Hermitian symmetric Lie algebra $\mathfrak{h}$, they revert respectively to $\widetilde A$ given by (\ref{tA-con}) and $\widetilde{K}$ given by (\ref{gauge5}) in $\mathfrak{h}$. Next, by using the prescribed curvature representation, we make the inverse gauge transformation to the connection ${\widetilde A}^{\prime}$ to obtain the corresponding connection 1-form $A^{\prime}$ and the curvature 2-form $K^{\prime}$. The equation coming from the prescribed curvature representation associated to $A^{\prime}$ and $K^{\prime}$ is just the model we pursue to find.

By comparing the recursion operator ${\widetilde {\mathcal{R}}}_{sym}$ with the generalized recursion operator ${\widetilde {\mathcal{R}}}_{red}$, we write down respectively the desired prescribed curvature 2-form $\widetilde{K}^{\prime}$ and the connection 1-form $\widetilde{A}^{\prime}$ on $\mathbb R^2\times U(n)$ as follows.
\begin{eqnarray}\label{Newtwo}
{\widetilde K}'&=&-\lambda^2(4\beta+\frac{1}{2}\alpha)[Q,[Q,ad_\omega^{-1}Q]_{\mathfrak k}]~dx\wedge dt
\end{eqnarray}
and
\begin{eqnarray}\label{Newone}
{\widetilde A}'&=&(\lambda\omega+Q)dx+\hat{V}_4'dt,
\end{eqnarray}
where, analogous to that of (\ref{V4hat}),
\begin{eqnarray}
\hat{V}_4'&=&\lambda^4\alpha\omega+\lambda^3 \alpha {V'}^{(1)}+\lambda^2\alpha {V'}^{(2)}+\lambda
  \Big(\alpha {{V'}}^{(3)}+(4\beta+\frac{1}{2}\alpha)
  (ad_\omega^{-1}[Q,[Q,ad_\omega^{-1}Q]_{\mathfrak k}]\nonumber\\
  &&+\partial_x^{-1}[Q,ad_\omega^{-1}[Q,[Q,ad_\omega^{-1}Q]_{\mathfrak k}]]_{\mathfrak k})\Big)
  +\Big(\alpha V^{'(4)}+(4\beta+\frac{1}{2}\alpha)M^{'}\Big),\label{V4}
  \end{eqnarray}
in which ${{V'}}^{(i)}={{V'}}_{\mathfrak m}^{(i)}+{{V'}}_{\mathfrak k}^{(i)} ~(i=1,2,3,4)$ are determined by the following iterated relations from ${V'}^{(0)}=\omega$,
\begin{flalign*}
[\omega,{V'}^{(i)}]=\widetilde{\mathcal{R}}_{red}[\omega,{V'}^{(i-1)}]
\end{flalign*}
that is to say,
\begin{flalign}
{{V'}}_{\mathfrak m}^{(i)} &= \text{ad}_\omega^{-1}\left(({{{V'}}_{\mathfrak m}^{(i-1)}})_x - [Q,{{V'}}_{\mathfrak k}^{(i-1)}] - [Q,{{V'}}_{\mathfrak m}^{(i-1)}]_{\mathfrak m}\right), \label{V'}\\
{{V'}}_{\mathfrak k}^{(i)} &= \partial_x^{-1}[Q,{{V'}}_{\mathfrak m}^{(i)}]_{\mathfrak k}, \label{V't}
\end{flalign}
and $M^{'}=M^{'}_{\mathfrak m}+M^{'}_{\mathfrak k}$ with
\begin{align}
M'_{\mathfrak m} &= \text{ad}_\omega^{-2}\left[Q,\left[Q,\text{ad}_\omega^{-1}Q\right]_{\mathfrak k}\right]_x - \text{ad}_\omega^{-1}\left[Q,\text{ad}_\omega^{-1}\left[Q,\left[Q,\text{ad}_\omega^{-1}Q\right]_{\mathfrak k}\right]\right]_{\mathfrak m} \nonumber\\
&\quad - \text{ad}_\omega^{-1}\left[Q,\partial_x^{-1}\left[Q,\text{ad}_\omega^{-1}\left[Q,\left[Q,\text{ad}_\omega^{-1}Q\right]_{\mathfrak k}\right]\right]_{\mathfrak k}\right],\label{M'} \\
M'_{\mathfrak k} &= \partial_x^{-1}\left[Q,M'_{\mathfrak m}\right]_{\mathfrak k}. \label{M't}
\end{align}
One will see in \S5 that the above connection $\widetilde{A}^{\prime}$ and curvature $\widetilde{K}^{\prime}$ given in (\ref{Newone}) and (\ref{Newtwo}) meet our requirement.
Moreover, it is a direct verification that the prescribed curvature representation
\begin{eqnarray}
F_{\widetilde A'}=d{\widetilde A'}-{\widetilde A'}\wedge{\widetilde A'}={\widetilde K}' \label{Newthree}
\end{eqnarray} gives that $Q$ satisfy
\begin{equation}
   Q_t=\Bigg(\left(\alpha {V'}^{(4)}+(4\beta+\frac{1}{2}\alpha)M'\right)_x-[Q,\alpha {V'}^{(4)}+(4\beta+\frac{1}{2}\alpha)M']\Bigg)_{\mathfrak m}. \label{Q4s}
\end{equation}
Eq.(\ref{Q4s}) is the counterpart of Eq.(\ref{gauge}) in the Hermitian reductive Lie algebra $\mathfrak g$. The explicit expression of ${V'}^{(4)}$ in of Eq.(\ref{Q4s}) will be given in the Appendix \ref{A.1} at the end of the paper. $M^{\prime}$ in Eq.(\ref{Q4s}) is given by (\ref{M'}) and (\ref{M't}).

One notes that when $\lambda=0$, (\ref{Newthree}) becomes a zero curvature representation, that is, there exists a unitary matrix $E\in U(n)$ fulfilling
\begin{eqnarray}
\left\{\begin{array}{c}E_x=QE,~~~~~~~~~~~~~~~~~~~~~~~~~~~~~~\\
E_t=\left(\alpha {V'}^{(4)}+(4\beta+\frac{1}{2}\alpha)M'\right)E.\end{array}\right.\label{Et}
\end{eqnarray}
We shall make a gauge transformation to the connection ${\widetilde A'}$ by taking $G=E$
\begin{eqnarray}
{\widetilde A'}\longmapsto{A'}=-G^{-1}dG+G^{-1}{\widetilde A'}G. \nn\label{gauge3}
\end{eqnarray}
Hence
\begin{eqnarray*}
 A' &=& \lambda{\widetilde\gamma} dx+(E^{-1}V^{'}_4E-E^{-1}E_t)dt\\
  &=&\lambda{\widetilde\gamma} dx+\Big(\lambda^4\alpha{\widetilde\gamma}+\lambda^3\alpha E^{-1}QE+\lambda^2\alpha E^{-1}V^{'(2)}E+\lambda\alpha E^{-1}V^{'(3)}E\\
  &&+(4\beta+\frac{1}{2}\alpha)E^{-1}(ad_\omega^{-1}[Q,[Q,ad_\omega^{-1}Q]_{\mathfrak k}]+\partial_x^{-1}[Q,ad_\omega^{-1}[Q,[Q,ad_\omega^{-1}Q]_{\mathfrak k}]]_{\mathfrak k})E\\
  && +\alpha E^{-1}V^{'(4)}E+(4\beta+\frac{1}{2}\alpha)E^{-1}(M_{\mathfrak m}'+M_{\mathfrak k}')E-E^{-1}E_t\Big)dt,
  \end{eqnarray*}
where $\widetilde \gamma=E^{-1}\omega E\in F_s^{\mu}(k_1,\cdots,k_s)\hookrightarrow u(n)$ and
\begin{eqnarray*}
 K' &=& E^{-1}{\widetilde K'}E=-\left(\lambda^2(4\beta+\frac{1}{2}\alpha)E^{-1}[Q,[Q,ad_\omega^{-1}Q]_{\mathfrak k}]E\right)~dx\wedge dt.
  \end{eqnarray*}
On the other hand, by a direct calculation and using the second equation in (\ref{Et}), we have
  \begin{eqnarray}
  F_{A'} &=&-\left( \lambda {\widetilde\gamma}_t-(E^{-1}\hat{V_{4}'}E-E^{-1}E_t)_x+[\lambda{\widetilde\gamma},E^{-1}\hat{V_{4}'}E-E^{-1}E_t]\right)dx\wedge dt\nonumber\\
  &=&-\bigg(\lambda^2(4\beta+\frac{1}{2}\alpha)\big(E^{-1}[Q,[Q,ad_\omega^{-1}Q]_{\mathfrak k}]E \big)\nonumber\\
  &&+\lambda\big(\widetilde \gamma_t-\alpha[\widetilde\gamma,E^{-1}V_{\mathfrak m}^{'(4)}E]-(4\beta+\frac{1}{2}\alpha)[\widetilde \gamma,E^{-1}M'_{\mathfrak m}E]\big)\bigg)dx\wedge dt.\label{FA'}
  \end{eqnarray}
Substituting (\ref{FA'}) into the formula $F_{\widetilde A'} = d\widetilde A'-\widetilde A'\wedge \widetilde A'=\widetilde K'$, we obtain the desired equation
\begin{equation}
\widetilde\gamma_t = \Bigl[\widetilde\gamma,\alpha(E^{-1}V_{\mathfrak m}^{'(4)}E)+\big(4\beta + \frac{1}{2}\alpha\big)(E^{-1}M_{\mathfrak{m}}'E)\Bigr],\label{GFM12}
\end{equation}
where $M'_{\mathfrak m}$ is given by (\ref{M'}). The following is an expression of Eq.(\ref{GFM12}) that separates the fourth-order derivative terms from the lower order derivative terms.
\begin{eqnarray*}
    {\widetilde\gamma}_t&=&\alpha\bigg[{\widetilde\gamma},\sum_{1\leq i<j \leq s}\frac{16}{(\mu_i-\mu_j)^4}\big(({\widetilde\gamma}_{xxxx})_{ij}+
    ({\widetilde\gamma}_{xxxx})_{ji}\big)\bigg]
    \nonumber\\
    &&+\bigg[{\widetilde\gamma},\alpha f({\widetilde\gamma},{\widetilde\gamma}_{x},{\widetilde\gamma}_{xx},{\widetilde\gamma}_{xxx})
    +\big(4\beta+\frac{1}{2}\alpha\big)g({\widetilde\gamma},{\widetilde\gamma}_{x},{\widetilde\gamma}_{xx},{\widetilde\gamma}_{xxx})\bigg],
\end{eqnarray*}
where $f({\widetilde\gamma},{\widetilde\gamma}_{x},{\widetilde\gamma}_{xx},{\widetilde\gamma}_{xxx})$ and $g({\widetilde\gamma},{\widetilde\gamma}_{x},{\widetilde\gamma}_{xx},{\widetilde\gamma}_{xxx})$
are given in Appendix \ref{A.2} at the end of the paper.

\smallskip
It is easy to verify that the inverse of the above process is also true. Hence, as a by-production in finding the model (\ref{GFM12}) in the Hermitian reductive Lie algebra ${\mathfrak g}=u(n)$, we have
\begin{Lemma}
Eqs.(\ref{Q4s}) and Eq.(\ref{GFM12}) in the Hermitian reductive Lie algebra $\mathfrak g=u(n)$ are gauge equivalent to each other.
\end{Lemma}

We now end this section with listing the three models that we pursue to find.
\begin{Theorem} \label{three-eq}
After taking the derivative of moving curves $\gamma$ with respect to $x$ ( $\widetilde \gamma=\gamma_x$), we have the following three basic models in the Hermitian reductive Lie algebra $\mathfrak{g}=u(n)$:
\begin{eqnarray}
\widetilde\gamma_t&=&-\sum\limits_{1\leq i<j \leq s}\frac{4}{(\mu_i-\mu_j)^2}\left[{\widetilde\gamma},({\widetilde\gamma}_{x})_{ij}+({\widetilde\gamma}_{x})_{ji}\right]_x,\label{GGLIE}\\
\widetilde{\gamma}_t &=& -\sum_{\substack{1\leq i<j \leq s}} \frac{4}{(\mu_i-\mu_j)^2}\left( (\widetilde{\gamma}_{xx})_{ij}
+ (\widetilde{\gamma}_{xx})_{ji} \right)_x + \frac{32\sqrt{-1}}{(\mu_i-\mu_j)^5}\left[[\widetilde{\gamma},
\widetilde{\gamma}_x]_{ij},[\widetilde{\gamma},\widetilde{\gamma}_x]_{ji} \right]_x\nonumber \\
&& - \sum_{\substack{1\leq i<j \leq s}} \sum_{\substack{1\leq m<n \leq s}} \frac{16}{(\mu_i-\mu_j)^2(\mu_m-\mu_n)^2}
\Bigg( \left[ \widetilde{\gamma}_x,[\widetilde{\gamma},\widetilde{\gamma}_x]_{ij}
+ [\widetilde{\gamma},\widetilde{\gamma}_x]_{ji}\right]_{mn}\nonumber\\
&& + \left[ \widetilde{\gamma}_x, [\widetilde{\gamma},\widetilde{\gamma}_x]_{ij}
+ [\widetilde{\gamma},\widetilde{\gamma}_x]_{ji} \right]_{nm} \Bigg)_x,\label{GGKDV} \\
{\widetilde\gamma}_t&=&\alpha\bigg[{\widetilde\gamma},\sum_{1\leq i<j \leq s}\frac{16}{(\mu_i-\mu_j)^4}\big(({\widetilde\gamma}_{xxxx})_{ij}+
    {\widetilde\gamma}_{xxxx})_{ji}\big)\bigg]
    \nonumber\\
    &&+\bigg[{\widetilde\gamma},\alpha f({\widetilde\gamma},{\widetilde\gamma}_{x},{\widetilde\gamma}_{xx},{\widetilde\gamma}_{xxx})
    +\big(4\beta+\frac{1}{2}\alpha\big)g({\widetilde\gamma},{\widetilde\gamma}_{x},{\widetilde\gamma}_{xx},{\widetilde\gamma}_{xxx})\bigg],\label{Ga4}
\end{eqnarray}
where $\widetilde \gamma\in {\rm F}_s^{\rm \mu}(k_1,\cdots,k_s)\hookrightarrow u(n)$.
\end{Theorem}

We point out that when $4\beta+\frac{1}{2}\alpha=0,$ Eq. (\ref{Ga4}) (i.e. Eq.(\ref{GFM12})) returns to the fourth equation in the Langer-Perline's LIE hierarchy. We will show in Section 5 that
Eqs.(\ref{GGLIE}),(\ref{GGKDV}) and (\ref{Ga4}) constitute the basic vortex models in the Hermitian reductive Lie algebra ${\mathfrak g}=u(n)$, up to the third-order approximation.

\section*{\S4. Equivalence}

In the previous section, we have introduced respectively the recursion operators $\widetilde{\mathcal{R}}_{red}$ and $\mathcal{R}_{red}$ in the Hermitian reductive Lie algebra $\mathfrak g=u(n)$. By using $\widetilde{\mathcal{R}}_{red}$, we have deduced the Fordy-Kulish's NLS hierarchy, i.e., $Q_t=\widetilde X^{'(L)}$ ($L\ge0$), and meanwhile, by using $\mathcal{R}_{red}$, we have derived the Langer-Perline's LIE hierarchy, i.e., ${\gamma}_t=X^{'(L)}$ ($L\ge0$). In this section we shall show that the $(L+2)$-th Fordy-Kulish's NLS equation $Q_t=\widetilde X^{'(L+2)}$ is equivalent to the $L$-th Langer-Perline's LIE equation ${\gamma}_t=X^{'(L)}$ in $\mathfrak g$. This generalizes greatly the famous Da Rios-NLS correspondence. It should be mentioned that in the Hermitian symmetric Lie algebra $\mathfrak h$, Langer and Perline in \cite{LP} just established a correspondence between the $(L+2)$-th Fordy-Kulish's NLS equation and the $L$-th Langer-Perline's LIE equation.

In the Hermitian reductive Lie algebra $\mathfrak g=\mathfrak u(n)$, the second to fourth equations in the Fordy-Kulish's NLS hierarchy read
\begin{eqnarray}
Q_t&=&\widetilde{X}^{(2)}=Q_x,\nonumber\\
Q_t&=&\widetilde X^{(3)}=ad_\omega^{-1}Q_{xx}-\frac{1}{2}[Q,[Q,ad_\omega^{-1}Q]_{\mathfrak k}]-[Q,ad_\omega^{-1}Q_x]_{\mathfrak m},\label{Q2}\\
Q_t&=&\widetilde X^{(4)}=ad_\omega^{-2}Q_{xxx}-(\frac{1}{2}ad_\omega^{-1}[Q,[Q,ad_\omega^{-1}Q]_{\mathfrak k}]+ad_\omega^{-1}[Q,ad_\omega^{-1}Q_x]_{\mathfrak m})_x\nonumber\\
&&-[Q,\partial_x^{-1}([Q,ad_\omega^{-2}Q_{xx}-\frac{1}{2}ad_\omega^{-1}[Q,[Q,ad_\omega^{-1}Q]_{\mathfrak k}]-ad_\omega^{-1}[Q,ad_\omega^{-1}Q_x]_{\mathfrak m}]_{\mathfrak k})]\nonumber\\
&&-[Q,ad_\omega^{-2}Q_{xx}-\frac{1}{2}ad_\omega^{-1}[Q,[Q,ad_\omega^{-1}Q]_{\mathfrak k}]-ad_\omega^{-1}[Q,ad_\omega^{-1}Q_x]_{\mathfrak m}]_{\mathfrak m}.\label{Q3}\\
\nonumber
\end{eqnarray}
Meanwhile, the first three equations in the Langer-Perline's LIE hierarchy are
\begin{eqnarray*}
{\gamma}_t&=&X^{'(0)}=\{\omega\}=T=\gamma_x,\\
{\gamma}_t&=&X^{'(1)}=\{Q\}=ad_{\{\omega\}}
^{-2}[{\gamma}_x,{\gamma}_{xx}]=-\sum\limits_{1\leq i<j \leq s}\frac{4}{(\mu_i-\mu_j)^2}([{\gamma}_x,{\gamma}_{xx}]_{ij}+[{\gamma}_x,{\gamma}_{xx}]_{ji}), \\
{\gamma}_t&=&X^{'(2)}=\{ad_\omega^{-1} Q_x+\frac{1}{2}[Q,ad_\omega^{-1}Q]_{\mathfrak k}\}\nonumber\\
&=&-\sum\limits_{1\leq i<j \leq s}\frac{4}{(\mu_i-\mu_j)^2}\big(({\gamma}_{xxx})_{ij}+({\gamma}_{xxx})_{ji}\big)+\frac{32\sqrt{-1}}{(\mu_i-\mu_j)^5}[[{\gamma}_x,{\gamma}_{xx}]_{ij},[{\gamma}_x,{\gamma}_{xx}]_{ji}]\nonumber\\
&&-\sum\limits_{1\leq i<j \leq s}\sum\limits_{1\leq m<n \leq s}\frac{16}{(\mu_i-\mu_j)^2(\mu_m-\mu_n)^2}\big([{\gamma}_{xx},[{\gamma}_x,{\gamma}_{xx}]_{ij}+[{\gamma}_x,{\gamma}_{xx}]_{ji}]_{mn}\nonumber\\
&&+[{\gamma}_{xx},[{\gamma}_x,{\gamma}_{xx}]_{ij}+[{\gamma}_x,{\gamma}_{xx}]_{ji}]_{nm}\big).\\
\end{eqnarray*}

\begin{Theorem} \label{cor}
For any $L\ge1$, the $(L+2)$-th Fordy-Kulish's NLS equation $Q_t=\widetilde X^{'(L+2)}$ in the Hermitian reductive Lie algebra $\mathfrak g=\mathfrak u(n)$ is equivalent to the $L$-th Langer-Perline's LIE equation ${\gamma}_t=X^{'(L)}$ in $\mathfrak g$.
\end{Theorem}

\begin{proof}
First of all, we recall from \S3 that the $L$-th Langer-Perline's LIE equation $\gamma_t=X^{'(L)}$ in the Hermitian reductive Lie algebra $\mathfrak g=\mathfrak u(n)$ can also be expressed by
    \begin{eqnarray*}
        \gamma_t=\{V^{'(L)}\},
    \end{eqnarray*}
where $\gamma\in \{E^{-1}\omega E~|~E\in U(n)\}\hookrightarrow u(n)$.
Taking the derivative with respect to $x$ in the both sides of the equation, we have
    $(\gamma_t)_x=\{V^{'(L)}\}_x$.
If we set $\gamma_x=\widetilde\gamma$, then the above equation becomes
\begin{eqnarray}
\widetilde \gamma_t=\{V^{'(L)}\}_x,\label{tilde}
\end{eqnarray}
where in the right-hand side of  Eq.(\ref{tilde}), $\gamma_x$ is replaced by $\widetilde \gamma$, $\gamma_{xx}$ by ${\widetilde \gamma}_x$ and so on.
Conversely, if $\widetilde \gamma$ satisfies Eq.(\ref{tilde}), we then set $\gamma=\int^x {\widetilde\gamma}(s,t)ds$. One sees that $\gamma$ satisfies
$\gamma_t=\{V^{'(L)}\}$. This indicates that $L$-th Langer-Perline's LIE equation $\gamma_t=X^{'(L)}$ is equivalent to Eq.(\ref{tilde}).

Next, we shall prove that the $(L+2)$-th Fordy-Kulish's NLS equation $Q_t=\widetilde X^{'(L+2)}$ in the Hermitian reductive Lie algebra $\mathfrak g=\mathfrak u(n)$ and  Eq.(\ref{tilde}) are gauge equivalent to each other. For this purpose, we set $\widetilde A'$ to be a connection 1-form on the trivial bundle $\mathbb{R}^2\times U(n)$ as follows
    \begin{eqnarray}
        \widetilde A'&=&(\lambda\omega+Q)~dx+V~dt,\label{0000}
    \end{eqnarray}
    where $V=\sum\limits_{j=0}^{L+1}V^{'(j)}\lambda^{L+1-j}$ is a polynomial ansatz in the spectral parameter $\lambda$, in which $V^{'(j)}$ ($j=0,1,\cdots,L+1$) are determined recursively by the relations (\ref{eq:1}) (\ref{eq:2}) and (\ref{eq:3}). $Q_t=\widetilde X^{'(L+2)}$ is equivalent to the zero-curvature representation:
    \begin{eqnarray*}
        F_{\widetilde A'}=d\widetilde A'-\widetilde A'\wedge \widetilde A'=0.
    \end{eqnarray*}
    Now, we choose a fundamental solution  $E=E(x,t)\in U(n)$ to
    $$
    \left\{
    \begin{array}{l}
    E_x = QE, \\[4pt]
    E_t = V^{'(L+1)}E
    \end{array}
    \right.
    $$
    at $\lambda=0$ and make a gauge transformation to the connection $\widetilde A'$ by  $G=E$:
    \begin{eqnarray}
        \widetilde A'\mapsto A'&=&-G^{-1}dG+G^{-1}\widetilde A'G=-E^{-1}dE+E^{-1}\widetilde A'E \nonumber\\
        &=&\lambda{\widetilde\gamma} dx+\bigg(\sum\limits_{j=0}^{L+1}E^{-1}V^{'(j)}(x,t)E\lambda^{L+1-j}-E^{-1}E_t\bigg)dt\label{WAA}
    \end{eqnarray}
    where $\widetilde\gamma=E^{-1}\omega E$.

    From the Yang-Mills theory, we know that $F_A'=E^{-1}F_{\widetilde A'}E=0$. A direct computation shows that
    \begin{eqnarray*}
        0&=&F_A'=dA'-A'\wedge A'\\
        &=&-\bigg(\lambda^{L+2}[\widetilde\gamma,E^{-1}V^{'(0)}E]+\sum\limits_{j=0}^{L-1}\lambda^{L+1-j}E^{-1}([\omega,V^{'(j+1)}]-V_x^{'(j)}\\
        &&+[Q,V^{'(j)}])E+\lambda\big(\widetilde\gamma_t-(E^{-1}V^{'(L)}E)_x+[\widetilde\gamma,E^{-1}V^{'(L+1)}E-E^{-1}E_t]\big)\\
        &&-(E^{-1}V^{'(L+1)}E-E^{-1}E_t)_x\bigg)dx\wedge dt\\
        &=&-\lambda\big(\widetilde\gamma_t-(E^{-1}V^{'(L)}E)_x\big)dx\wedge dt.
    \end{eqnarray*}
    Therefore, the $(L+2)$-th Fordy-Kulish's NLS equation $Q_t=\widetilde X^{'(L+2)}$ is gauge transformed to
    \begin{equation}
        \widetilde\gamma_t=(E^{-1}V^{'(L)}E)_x=\{V^{'(L)}\}_x,\nonumber
    \end{equation}
which is exactly Eq.(\ref{tilde}).

Conversely, if $\widetilde\gamma=E^{-1}\omega E$ solves Eq.(\ref{tilde}), i.e. $\widetilde\gamma_t=(E^{-1}V^{'(L)}E)_x$, where $E=E(x,t)\in U(n)$. Without loss of generality, as indicated in Section 3, $E$ is assumed to satisfy $E_x=QE$ for some $Q(x,t)\in \mathfrak m$. We then make the following gauge transformation to the connection $A'$ by using $G=E^{-1}$:
    \begin{eqnarray}
        A'\mapsto\widetilde A'&=&-G^{-1}dG+G^{-1} A'G=dE\cdot E^{-1}+EA'E^{-1}\nonumber\\
        &=&(\lambda\omega+Q)dx+(V'_{L+1}-V^{'(L+1)}+E_t E^{-1})dt,\label{WA}
    \end{eqnarray}
    where $E_tE^{-1}$ is independent of $\lambda$. By a direct calculation, we have
    \begin{eqnarray}
        0&=&F_{\widetilde A'}-\widetilde A'\wedge \widetilde A' \nonumber\\
        &=&-\big(Q_t-(V'_{L+1}-V^{'(L)}+E_tE^{-1})_x+[\lambda\omega+Q,V'_{L+1}-V^{'(L)}\nonumber\\
        &&+E_tE^{-1}]\big)dx\wedge dt\nonumber\\
        &=&-\bigg(\lambda(V_x^{'(L)}+[Q,V^{'(L)}]+[\omega,E_tE^{-1}])+(Q_t+[Q,E_tE^{-1}]\nonumber\\
        &&-(E_tE^{-1})_x)\bigg)dx\wedge dt. \label{F_A}
    \end{eqnarray}
    The vanishing of the coefficients of $\lambda^1$ and $\lambda^0$ in (\ref{F_A}) implies that
    $$
    \left\{
    \begin{array}{l}
    [\omega,E_tE^{-1}]= -(V^{'(L)}_x+[Q,V^{'(L)}])=-[\omega,V_{\mathfrak m}^{'(L+1)}], \\[4pt]
    (E_tE^{-1})_x = Q_t+[Q,E_tE^{-1}]
    \end{array}
    \right.
    $$
    From the first equation in the above differential system, we see that
    $$(E_tE^{-1})_{\mathfrak m}=-V_{\mathfrak m}^{'(L+1)}.$$
    Comparing the $\mathfrak k$-components in both sides of the second equation, we obtain
    $$
    {(E_tE^{-1})_{\mathfrak k}}_x=[Q,(E_tE^{-1})_{\mathfrak m}]_{\mathfrak k}
    $$
    which implies $(E_tE^{-1})_{\mathfrak k}=(\partial_x^{-1}[Q,-V^{'(L+1)}])_{\mathfrak k}=-V^{'(L+1)}_{\mathfrak k}+F(t)$ for some $F(t)\in\mathfrak k$ that is independent of $x$. Hence we have
    \begin{equation}
        E_tE^{-1}=(E_tE^{-1})_{\mathfrak m}+(E_tE^{-1})_{\mathfrak k}=-V^{'(L+1)}+F(t).\label{E_tE-1}
    \end{equation}
    In order to cancel $F(t)$ in (\ref{E_tE-1}), we modify $E$ by
    $$E\to \hat E=G(t)E.$$
    where $G(t)\in \mathfrak k$ depends only on $t$ and satisfies $\frac{dG}{dt}(t)+G(t)F(t)=0$. Obviously, such a $G$ exists. One may verify straightforwardly that $\hat E_t=G_t(t)E+G(t)E_t=-G(t)F(t)E+G(t)(-V^{'(L+1)}+F(t))E=-G(t)V^{'(L+1)}G^{-1}(t)\hat E$, and hence the second term on the right-hand-side of (\ref{E_tE-1}) vanishes identically if we replace $E$ by $\hat E$ in the connection $ A'$ given by (\ref{WAA}). The $\widetilde A'$ obtained in (\ref{WA}) is exactly the connection 1-form given in (\ref{0000}) with $Q$ being replaced by $\widetilde Q=G(t)QG^{-1}(t)$. This shows that Eq.(\ref{tilde}) is gauge transformed to $Q_t=\widetilde X^{'(L+2)}$. The proof that the $(L+2)$-th Fordy-Kulish's NLS equation $Q_t=\widetilde X^{'(L+2)}$  is gauge equivalent to Eq.(\ref{tilde}) is completed.

Since $L$-th Langer-Perline's LIE equation $\gamma_t=X^{'(L)}$ is equivalent to Eq.(\ref{tilde}), we thus have proved the equivalence between the $(L+2)$-th Fordy-Kulish's NLS equation $Q_t=\widetilde X^{'(L+2)}$ and the $L$-th Langer-Perline's LIE equation ${\gamma}_t=X^{'(L)}$.
\end{proof}

According to the terminology in \cite{LP}, the motion model of the $L$-th Langer-Perline's LIE equation ${\gamma}_t=X^{'(L)}$ in $\mathfrak g$ is a geometric realization of the $(L+2)$-th Fordy-Kulish's NLS equation $Q_t=\widetilde X^{'(L+2)}$. The geometric realization of the generalized Schr\"odinger equation (\ref{3-wave}) in $\mathfrak u(3)$ is described in detail in \cite{DDZ}.

\section*{\S5. Reduction correspondence}
In \cite{FK}, Fordy and Kulish pointed out roughly the reduction between the complex flag manifold $F^{\mu}_3(1,1,1)=\frac{SU(3)}{S(U(1)\times U(1)\times U(1))}$ and the complex symmetric space $\mathbb CP^2=\frac{SU(3)}{S(U(1)\times U(2))}$:
\begin{eqnarray}
\frac{U(3)}{U(1)\times U(2)}\equiv \frac{SU(3)}{S(U(1)\times U(2))}\hookrightarrow \frac{SU(3)}{S(U(1)\times U(1)\times U(1))}\equiv \frac{U(3)}{U(1)\times U(1)\times U(1)}\nonumber
\end{eqnarray}
and its fascinating consequence that the $3$-rd Fordy-Kulish's NLS equation (\ref{3-wave}) relating to $F^{\mu}_3(1,1,1)$ in ${\mathfrak g}=\mathfrak u(3)$ deduces to the 2-vector nonlinear Schr\"odinger equation relating to $\mathbb CP^2$ in $\mathfrak h=\mathfrak u(3)$.
In this section, we shall explore the mechanics behind this reduction in the framework of Hausdorff convergence in Riemannian manifolds.

Let's first review some basic relevant facts about the Hausdorff distance $d_H$ and the $d_H$-convergence in $\mathfrak u(n)$. With the canonical Killing metric $\langle X,Y\rangle =-\hbox{tr}(XY)$ on $\mathfrak u(n)$, it is well known that the Hausdorff distance between a flag orbit space ${\rm F}_s^{\rm \mu}(k_1,\cdots,k_s)$ ($s\ge3$, $k_1+\cdots+k_s=n$) and a symmetric space ${\rm F}_2^{(1,-1)}(k,n-k)=Gr(k,n-k)$ is defined by
\begin{eqnarray*}
d_H(F^{\rm \mu}_s(k_1,\cdots,k_s),Gr(k,n-k))&=& max\Bigl\{\sup\limits_{X\in Gr(k,n-k)}dist(X,F^{\rm \mu}_s(k_1,k_2,\dots,k_s)),\\
&&\sup\limits_{Y\in F^{\rm \mu}_s(k_1,k_2,\cdots,k_s)}dist(Y,Gr(k,n-k))\Bigr\}.
\end{eqnarray*}
For a family of Hermitian reductive Lie algebra $\mathfrak g=\mathfrak u(n)$ given by (\ref{decomp}) with $s\ge3$ parameterized by $\mu$, if there exists an integer $\nu$ such that $k=k_1+\cdots+k_{\nu}$, $n-k=k_{\nu+1}+\cdots+k_s$, and $\mu_1,\mu_2,\cdots,\mu_{\nu} \to 1$, $\mu_{\nu+1},\mu_{\nu+2},\cdots,\mu_{s} \to -1$,  then we have
$$\omega = \frac{\sqrt{-1}}{2}
\begin{pmatrix}
\mu_1 I_{k_1} & & & \\
& \mu_2 I_{k_2} & & \\
& & & \ddots & \\
& & & & \mu_s I_{k_s}
\end{pmatrix}~\to~ \sigma_3=\frac{\sqrt{-1}}{2}\left(\begin{array}{cc}
I_k&0\\0&-I_{n-k}\end{array}\right)$$
and denote it by $\omega\to \sigma_3$.
For the purpose of simplicity, we set $\mu_i=1+i\Delta\mu$ ($i=1,\cdots,\nu$) and $\mu_{i}=-1+i\Delta\mu$ ($i=\nu+1,\cdots,s$), where $\Delta\mu$ is an infinitesimal real parameter. Such choice of $\mu_{i}$ guarantees that $\mu_i$ ($i=1,2,\cdots,s$) are not only different, but also are equi-convergence as $\Delta \mu\to0$. We will always make this convention throughout the section unless otherwise declaration and call this convergence as the equi-convergence.

\begin{Theorem} \label{cv}
For the $\mu$-family of complex flag manifolds ${\rm F}^{\rm\mu}_s(k_1,\cdots,k_s)=\{E^{-1}\omega E|E\in U(n)\}$ with parameters $\mu_i=1+i\Delta\mu$ ($i=1,\cdots,\nu$) and $\mu_{j}=-1+i\Delta\mu$ ($i=\nu+1,\cdots,s$), we have that
\begin{eqnarray}
\lim_{\Delta\mu\to0}\hbox{\rm F}^{\rm \mu}_s(k_1,\cdots,k_s)
= \hbox{Gr}(k,n-k)\label{lim}
\end{eqnarray}
in the sense of Hausdorff convergence.
\end{Theorem}
\begin{proof}
    We shall compute the Hausdorff distance between ${\rm F}^{\rm\mu}_s(k_1,\cdots,k_s)$ and $Gr(k,n-k)$ by applying the formula
    \begin{eqnarray*}
    d_{H}({\rm F}^{\rm\mu}_s(k_1,k_2,\dots,k_s),Gr(k,n-k))& = & max\{\sup\limits_{X\in Gr(k,n-k)}dist(X,{\rm F}^{\rm\mu}_s(k_1,k_2,\dots,k_s))\\
    &&\sup\limits_{Y\in {\rm F}^{\rm\mu}_s(k_1,k_2,\cdots,k_s)}dist(Y,Gr(k,n-k))\}.
    \end{eqnarray*}
For $X\in Gr(k,n-k)$, we have
    \begin{eqnarray*}
    dist(X,{\rm F}^{\rm\mu}_s(k_1,k_2,\cdots,k_s))& = &\inf\limits_{Y\in {\rm F}^{\rm\mu}_s(k_1,k_2,\cdots,k_s)}dist(X,Y) \\
    & = &\inf\limits_{Y\in {\rm F}^{\rm\mu}_s(k_1,k_2,\cdots,k_s)}\sqrt{-2ntr(X-Y,X-Y)}\\
    & = &\inf\limits_{E_2\in U(n)}\sqrt{-2ntr{(E_1^{-1}\omega E_1-E_2^{-1}\sigma_3E_2)}^2}\\
    & \leq &\sqrt{-2ntr{(E_1^{-1}\omega E_1-E_1^{-1}\sigma_{3}E_1)}^2}\\
    & = & \sqrt{{\frac{n}{2}}(\sum\limits_{i=1}^\nu k_i(\mu_i-1)^2+\sum\limits_{i=\nu+1}^s k_i(\mu_i+1)^2)}\\
    & = &\sqrt{\frac{n}{2}\sum\limits_{i=1}^s k_i i^2}~~|\Delta\mu|.
    \end{eqnarray*}
One sees that the right-hand-side of the above last equality is independent of $X\in Gr(k,n-k)$. Similarly, for $Y\in {\rm F}^{\rm\mu}_s(k_1,k_2,\cdots,k_s)$, we also have
\begin{eqnarray*}
    dist(Y,Gr(k,n-k))& = &\inf\limits_{X\in Gr(k,n-k)}dist(X,Y)
    \le \sqrt{\frac{n}{2}\sum\limits_{i=1}^s k_i i^2}~~|\Delta\mu|.
    \end{eqnarray*}
Hence
\begin{eqnarray*}
    d_{H}({\rm F}^{\rm\mu}_s(k_1,k_2,\dots,k_s),Gr(k,n-k))& \le & \sqrt{\frac{1}{2}n\sum\limits_{i=1}^s k_i i^2}~~|\Delta\mu|,
    \end{eqnarray*}
which implies (\ref{lim}) and the proof of Theorem \ref{cv} is completed.
\end{proof}

The equi-convergence of $\{\hbox{\rm F}^{\rm \mu}_s(k_1,\cdots,k_s)\}$ to $\hbox{Gr}(k,n-k)$ described in Theorem \ref{cv} is also called that $\{\hbox{\rm F}^{\rm \mu}_s(k_1,\cdots,k_s)\}$ equi-collapses to $\hbox{Gr}(k,n-k)$, or simply, $\hbox{\rm F}^{\rm \mu}_s(k_1,\cdots,k_s)$ equi-collapses to $\hbox{Gr}(k,n-k)$. Now we run into a natural question, that is, when $F^{\rm \mu}_s(k_1,\dots,k_s)$ equi-collapses to the symmetric space $Gr(k,n-k)$, do the three basic models (\ref{GGLIE}), (\ref{GGKDV}) and (\ref{Ga4}) in $\mathfrak g$ revert respectively to the vortex equations (\ref{GLIE1}), (\ref{GKDV1}) and (\ref{GFM11}) in $\mathfrak h$? If this is the case, then the three models are exactly the leading-order, second-order and third-order vortex models in the reductive Lie algebra ${\mathfrak g}=\mathfrak u(n)$.
In order to answer this question, we should return to Lie algebras. Let's denote the decomposition of ${\mathfrak g}=\mathfrak u(n)$ corresponding to $F^{\rm \mu}_s(k_1,\dots,k_s)$ by $\mathfrak{g}={\mathfrak m'+\mathfrak k'}$, and the decomposition of ${\mathfrak h}=\mathfrak u(n)$ corresponding to $Gr(k,n-k)$ by $\mathfrak{h=m+k}$. One notes that in this case, $\mathfrak m\subset \mathfrak m'$, $\mathfrak k'\subset \mathfrak k$. Hence
the subspace $\mathfrak{m'}$ in the decomposition  $\mathfrak{g=m'+k'}$ is divided into two parts:
${\mathfrak m}_1={\mathfrak m'}|_{\mathfrak m}=\mathfrak m$ and ${\mathfrak m}_2={\mathfrak m'}|_{\mathfrak k}$. So we have
$\mathfrak{h} = \mathfrak{m} + \mathfrak{k} = \mathfrak{m} + (\mathfrak{m}_2 + \mathfrak{k}')$ and
$\mathfrak{g} = \mathfrak{m}' + \mathfrak{k}' = (\mathfrak{m}_1 + \mathfrak{m}_2) + \mathfrak{k}'$. Additionally, one has that
\begin{align*}
&[\mathfrak m_1,\mathfrak m_1] \subset \mathfrak{k} = \mathfrak{k}' + \mathfrak{m}_2,~
[\mathfrak m_2,\mathfrak m_2] \subset \mathfrak{k} = \mathfrak{k}' + \mathfrak{m}_2,~
[\mathfrak m_1, \mathfrak m_2] \subset \mathfrak{m}_1.
\end{align*}

Based on the above considerations, we come again to the Lax pair (\ref{05}) in the Hermitian reductive Lie algebra $\mathfrak g$, where the potential function $Q\in \mathfrak m^{\prime}$ depends on $x$, $t$ and, probably, on the parameters $\mu$, $V'_L=\sum\limits_{j=0}^{L}V^{'(j)}(x,t)\lambda^{L-j}\in \mathfrak g$ is the polynomial ansatz. Below in this section, we just focus on smooth solutions to Eq.(\ref{05}) and also smooth solutions to the associated integrable equation $Q_t=\widetilde{X}^{'(L)}$, unless otherwise stated. Let
$Q=Q(x,t;\mu)$ be a smooth solution to the $(L+1)$-th Fordy-Kulish's NLS equation $Q_t=\widetilde{X}^{'(L)}$ and $E=E(x,t,\lambda;\mu)\in U(n)$ be a smooth solution to Eq.(\ref{05}).
We claim that when $\Delta\mu \to 0$, the limit matrix $\widetilde E(x,t,\lambda)\in U(n)$ of $E=E(x,t,\lambda;\mu)$ satisfies $\widetilde E_x=(\lambda\sigma_3+P)\widetilde E$, where
\begin{eqnarray}P=\lim_{\Delta\mu\to0}Q\in \mathfrak m.\label{QP}\end{eqnarray}
In fact, because of the compactness of $U(n)$, $\lim\limits_{\Delta\mu\to0}E=\widetilde E\in U(n)$ exists obviously for any fixed $(x,t,\lambda)$. Furthermore, since the smoothness of $E$ and $x$ is independent of $\mu$, we also have that $\lim\limits_{\Delta\mu\to0}E_x=\widetilde E_x$ and $\widetilde E_x=(\lambda\sigma_3+P)\widetilde E$ with $P=\lim\limits_{\Delta\mu\to0}Q$. What remains us to do is to show that $P\in \mathfrak m$. Indeed, we see that when $\Delta\mu \to 0$, the ${\rm F}^{\rm \mu}_s(k_1,\cdots,k_s)$ equi-collapse to $Gr(k,n-k)$, the infinitesimal version of this convergence is: $\mathfrak {g=k'\oplus m'} \to \mathfrak {h=k\oplus m}$. Especially, the tangent space $\mathfrak m'$ of ${\rm F}^{\rm \mu}_s(k_1,\cdots,k_s)$ goes to the tangent space $\mathfrak m$ of $Gr(k,n-k)$ as $\Delta\mu \to 0$. This means that
$$
Q=E_xE^{-1}|_{\lambda=0}\in \mathfrak m'\to \widetilde E_x\widetilde E^{-1}|_{\lambda=0}=P\in \mathfrak {m} ~~\hbox{as}~~\Delta\mu\to0.
$$
Hence (\ref{QP}) is valid and the claim is verified. One also notes that, because $Q(x,t;\mu)$ is also smooth in $\mu$, (\ref{QP}) is equivalent to (by the mean value theorem)
\begin{equation}
Q=Q(x,t;\mu)=P(x,t)+C(x,t;\mu)\cdot \Delta\mu, \label{infinity}
\end{equation}
where $C(x,t;\mu)=\sum_{j=1}^{\nu}j\frac{\partial Q}{\partial \mu_j}(x,t;1+j\theta_j\Delta\mu)+\sum_{j=\nu+1}^{s}j\frac{\partial Q}{\partial \mu_j}(x,t;-1+j\theta_j\Delta\mu)$ for some $\theta_j$ with $0<\theta_j<1$ ($1\le j\le s $).
\begin{Lemma} \label{lem}
For the above smooth potential functions $Q(x,t;\mu)$ and $P(x,t)$, if we denote $Q(x,t;\mu)=Q_1(x,t;\mu)+Q_2(x,t;\mu)$ with $Q_i(x,t;\mu)=Q(x,t;\mu)|_{\mathfrak m_i}\in\mathfrak m_i$ ($i=1,2$), then we have

(a) For $\forall \alpha_1,\alpha_2\ge0 ~\hbox{with}~\alpha_1,\alpha_2\in \mathbb N$,
\begin{eqnarray}
\partial^{\alpha_1}_x \partial^{\alpha_2}_tQ(x,t;\mu)~~\rightarrow~~\partial^{\alpha_1}_x \partial^{\alpha_2}_t P(x,t),\label{part}
\end{eqnarray}

(b)
\begin{eqnarray}
\lim_{\Delta \mu \to 0} \frac{Q_2(x,t;\mu)}{\Delta\mu}=C'_2(x,t),
\end{eqnarray}
for some $\mathfrak m_2$-valued smooth function $C'_2(x,t)$.
\end{Lemma}
\begin{proof}
(a)
Since $x$ and $t$ are independent of $\mu$ and the $\mathfrak g$-valued functions $Q(x,t;\mu)$, $P(x,t)$ and $C(x,t;\mu)$ in (\ref{infinity}) are smooth, taking derivatives with respect to $x$ or $t$ to any given order on both sides of (\ref{infinity}), we then have (a).

(b) Restricting equation (\ref{infinity}) to $\mathfrak m_2$-part, we obtain
\begin{equation}
Q_2(x,t;\mu)=C_2(x,t;\mu)\cdot \Delta\mu, \label{C2}
\end{equation}
where $C_2(x,t;\mu)$ is the restriction of $C(x,t;\mu)$ to $\mathfrak m_2$. Hence
\begin{eqnarray}
\lim_{\Delta\mu \to 0}\frac{Q_2(x,t;\mu)}{\Delta\mu}&=&\lim_{\Delta\mu \to 0}C_2(x,t;\mu)\nn\\
&=&C_2'(x,t),
\end{eqnarray}
for some smooth $\mathfrak m_2$-valued function $C'_2(x,t)$.
\end{proof}

\begin{Lemma}\label{Le1}
     When $\Delta\mu\to0$, or in other words, $\omega$ equi-converges to $\sigma_3$, we have that smooth solutions to the $L$-th Langer-Perline's LIE equation ${\gamma}_t=X^{'(L)}$ in the Hermitian reductive Lie algebra
     $\mathfrak g=\mathfrak u(n)$ revert to smooth solutions to the equation $\gamma_t=X^{(L)}$  in the symmetric Lie algebra $\mathfrak h=\mathfrak u(n)$ for any given integer $L\ge1$.
\end{Lemma}
\begin{proof} From Theorem \ref{cor}, we only need to prove that, when $\Delta\mu\to0$, smooth solutions to the equation $Q_t=\widetilde{X}^{'(L)}$ in the Hermitian reductive Lie algebra
$\mathfrak g=\mathfrak u(n)$ revert to smooth solutions to the equation $P_t={\widetilde X}^{(L)}$ in the symmetric Lie algebra $\mathfrak h=\mathfrak u(n)$ for any given integer $L\ge1$. We shall prove it by induction. In fact, we first consider the case $L=1$. When $\Delta\mu\to0$, we know that ${\rm F}^{\rm\mu}_s(k_1,\cdots,k_s)$ equi-collapses to $Gr(k,n-k)$ and the validity of Lemma 2 imply that
$$
Q_t\to P_t~~\hbox{and}~~[\omega,Q]\to [\sigma_3,P],~~\hbox{hence}~~Q_t={\widetilde X}^{'(1)}\to P_t={\widetilde X}^{(1)}.
$$
Lemma 3 is true for $L=1$. We now set $Q_1=Q|_{{\mathfrak m}_1}$ ,$Q_2=Q|_{{\mathfrak m}_2}$, ${{\widetilde X}}_1^{'(k)}={{\widetilde X}}^{'(k)}|_{{\mathfrak m}_1}$ , ${{\widetilde X}}_2^{'(k)}={{\widetilde X}}^{'(k)}|_{{\mathfrak m}_2}$ for a general $k\in \mathbb N$.
Assuming that Lemma 3 is true for $L=k\ge1$, in other words, $Q_t={\widetilde X}^{'(k)} \to P_t={\widetilde X}^{(k)}$ as $\Delta\mu\to0$, i.e.
\begin{eqnarray}
(Q_1)_t={\widetilde X}_1^{'(k)}\to P_t={\widetilde X}^{(k)},      \quad (Q_2)_t={\widetilde X}_2^{'(k)}\to 0=0, \label{k}
\end{eqnarray}
we shall show that Lemma 3 is  true for $L=k+1$. In fact, by noting that
\begin{eqnarray*}
    {\widetilde X}^{'(k+1)}&=&{\widetilde {\mathcal{R}}}_{red}{\widetilde X}^{'(k)}=(\partial_x-ad_{Q}\partial_x^{-1}ad_{Q}|_{{\mathfrak k}'}-ad_{Q}|_{{\mathfrak m}'})ad_{\omega}^{-1}{\widetilde X}^{'(k)}\\
    &=& (\partial_x-ad_{Q_1+Q_2}\partial_x^{-1}ad_{Q_1+Q_2}|_{{\mathfrak k}'}-ad_{Q_1+Q_2}|_{{\mathfrak m}'})ad_{\omega}^{-1}({\widetilde X}^{'(k)}_1+{\widetilde X}^{'(k)}_2)\\
    &=&\partial_x(ad_{\omega}^{-1}({\widetilde X}^{'(k)}_1+{\widetilde X}^{'(k)}_2))-[Q_1+Q_2,\partial_x^{-1}([Q_1+Q_2,ad_{\omega}^{-1}({\widetilde X}^{'(k)}_1+{\widetilde X}^{'(k)}_2)]_{{\mathfrak k}'})]\\
    &&-[Q_1+Q_2,ad_{\omega}^{-1}({\widetilde X}^{'(k)}_1+{\widetilde X}^{'(k)}_2)]_{{\mathfrak m}'}
    \end{eqnarray*}
and restricting the above identity to the parts ${\mathfrak m}_1$ and ${\mathfrak m}_2$ respectively, we obtain
    \begin{eqnarray}
    (Q_1)_t&=&{\widetilde X}_1^{'(k+1)}=({\widetilde {\mathcal{R}}_{red}^{'(k)}}{\widetilde X}^{'(k)})|_{\mathfrak m_1} \nn\\
    &=& \partial_x ad_{\omega}^{-1}{\widetilde X}^{'(k)}_1-[Q_1,\partial_x^{-1}([Q_1,ad_{\omega}^{-1}{\widetilde X}^{'(k)}_1]_{{\mathfrak k}'}\nonumber+[Q_2,ad_{\omega}^{-1}{\widetilde X}^{'(k)}_2]_{{\mathfrak k}'})]\nn\\
    &&-[Q_1,ad_{\omega}^{-1}{\widetilde X}^{'(k)}_2]_{{\mathfrak m}_1}-[Q_2,ad_{\omega}^{-1} {\widetilde X}^{'(k)}_1]_{{\mathfrak m}_1},\label{R1}\\
    (Q_2)_t&=&{\widetilde X}_2^{'(k+1)}=({\widetilde {\mathcal{R}}_{red}^{'(k)}}{\widetilde X}^{'(k)})|_{\mathfrak m_2}\nn \\
    &=& \partial_x ad_{\omega}^{-1}{\widetilde X}^{'(k)}_2-[Q_2,\partial_x^{-1}([Q_1,ad_{\omega}^{-1}{\widetilde X}^{'(k)}_1]_{{\mathfrak k}'}+[Q_2,ad_{\omega}^{-1}{\widetilde X}^{'(k)}_2]_{{\mathfrak k}'})]\nonumber \\
    &&-[Q_1,ad_{\omega}^{-1}{\widetilde X}^{'(k)}_1]_{{\mathfrak m}_2}-[Q_2,ad_{\omega}^{-1}{\widetilde X}^{'(k)}_2]_{{\mathfrak m}_2}.\label{R2}
    \end{eqnarray}
As $\Delta\mu\to0$, we have $Q_t=(Q_1)_t+(Q_2)_t\rightarrow P_t$ from Lemma 2, that is to say,
\begin{eqnarray}
(Q_1)_t\rightarrow P_t,\quad (Q_2)_t\rightarrow 0. \label{tend}
\end{eqnarray}
Substituting the relations (\ref{k}) and (\ref{tend}) into (\ref{R2}) and by Lemma \ref{lem},  we have
    $$
    \partial_x ad_{\omega}^{-1}{\widetilde X}^{'(k)}_2-[Q_1,ad_{\sigma_3}^{-1}{\widetilde X}_1^{'(k)}]_{{\mathfrak m}_2}\to 0$$
which implies:
    $$ad_{\omega}^{-1}{\widetilde X}^{'(k)}_2\rightarrow\partial_x^{-1}([Q_1,ad_{\sigma_3}^{-1}{\widetilde X}_1^{'(k)}]_{{\mathfrak m}_2}).$$
Substituting this relation into the righthand-side of Eq.(\ref{R1}) and using the inductive hypothesis, we obtain that when $\Delta\mu\to0$,
    \begin{eqnarray*}
    ({\widetilde X}^{'(k+1)})|_{\mathfrak m_1}&=&({\widetilde {\mathcal{R}}_{red}}{\widetilde X}^{'(k)})|_{\mathfrak m_1} \\
    &\to& \partial_x ad_{\sigma_3}^{-1}{\widetilde X}^{'(k)}_1-[Q_1,\partial_x^{-1}([Q_1,ad_{\sigma_3}^{-1}{\widetilde X}^{'(k)}_1]_{{\mathfrak k}'})]-[Q_1,ad_{\omega}^{-1}{\widetilde X}^{'(k)}_2]_{{\mathfrak m}'}\\
    &\to & \partial_x ad_{\sigma_3}^{-1}{\widetilde X}_1^{'(k)}-[P,\partial_x^{-1}([P,ad_{\sigma_3}^{-1}{\widetilde X}_1^{'(k)}]_{{\mathfrak k}'})]-[P,\partial_x^{-1}([P,ad_{\sigma_3}^{-1}{\widetilde X}_1^{(k)}]_{{\mathfrak m}_2})]\\
    &\to & \partial_x ad_{\sigma_3}^{-1}{\widetilde X}^{(k)}-[P,\partial_x^{-1}([P,ad_{\sigma_3}^{-1}{\widetilde X}^{(k)}]_{\mathfrak {m}_2+\mathfrak k'=\mathfrak k})]\\
    &\overset{ad^{-1}_{\sigma_3}=-ad_{\sigma_3}}{=}&\widetilde{\mathcal{R}}_{sym}\widetilde{X}^{(k)}=\widetilde{X}^{(k+1)}
    \end{eqnarray*} and
    \begin{eqnarray*}
    {\widetilde X}^{'(k+1)}_2\to 0.
    \end{eqnarray*}
    Hence, we have that when $\Delta\mu\to0$,
    $$
    Q_t={\widetilde X}^{'(k+1)}\to P_t={\widetilde X}^{(k+1)}.
    $$
This completes the proof of Lemma \ref{Le1}.
\end{proof}

\begin{Lemma}\label{Le2}
When $F_s^{\mu}(k_1,\cdots,k_s)$ equi-collapses to $Gr(k,n-k)$, smooth solutions to Eq. (\ref{Ga4}) (i.e. Eq.(\ref{GFM12})) in the Hermitian reductive Lie algebra $\mathfrak g=\mathfrak u(n)$ revert to smooth solutions to Eq.(\ref{GFM11}) in the Hermitian symmetric Lie algebra $\mathfrak h=\mathfrak u(n)$.
\end{Lemma}
\begin{proof}
Based on Lemma 1, we need only to show that smooth solutions to the model (\ref{Q4s}) in the reductive Lie algebra $\mathfrak g=\mathfrak u(n)$ revert to smooth solutions to the model (\ref{gauge}) in the symmetric Lie algebra $\mathfrak h=\mathfrak u(n)$, when $F_s^{\mu}(k_1,\cdots,k_s)
$ equi-collapses to $Gr(k,n-k)$.
For this purpose, let's revisit the prescribed curvature representation for Eq.(\ref{Q4s}). Recall that
$${\widetilde A}'=(\lambda\omega+Q)dx+\hat{V_4'}dt,~~
{\widetilde K}'=-\lambda^2(4\beta+\frac{1}{2}\alpha)[Q,[Q,ad_\omega^{-1}Q]_{\mathfrak k}]~dx\wedge dt$$
are the 1-form connection and the 2-form prescribed curvature on the trivial bundle $\mathbb R^2\times U(n)$ given respectively by (\ref{Newone}) and (\ref{Newtwo}), where $Q\in \mathfrak m$ and $\hat{V_4'}$ is given by (\ref{V4}).
When $\Delta \mu\to0$, i.e. $F_s^{\mu}(\mu_1,\cdots,\mu_s)$ equi-collapses to $G(k,n-k)$, from $(b)$ of Lemma \ref{lem} we see that $\lim\limits_{\Delta\mu\to0} ad_\omega^{-1}Q$ exists for any fixed $(x,t)$. Now we pay our attention to the limits of ${\hat V}'_4$ and ${\widetilde K}'$. What remains for us to do is to show that $\lim\limits_{\Delta\mu\to0}{\hat V}'_4=V'_4$ and $\lim\limits_{\Delta\mu\to0}{\widetilde K}'=K'$. In fact, for $\hat V'_4$ , we see that when $\Delta \mu\to0$, the limits of the coefficients of $\lambda^i$ ($i=4,3,2,1,0$) are
 \begin{itemize}
    \item[] $\lambda^4$: ${V'}^{(4)}=\omega \to \sigma_3={V}^{(4)}$;
    \item[] $\lambda^3$: ${V'}^{(3)}=Q \to P={V}^{(3)}$;
    \item[] $\lambda^2$: ${V'}^{(2)}=ad_{\omega}^{-1}Q_x+\frac{1}{2}[Q,ad_{\omega}^{-1}Q]_{\mathfrak k}\to -[\sigma_3,P_x]+\frac{1}{2}[P,[P,\sigma_3]]={V}^{(2)}$;
    \item[] $\lambda^1$: \begin{eqnarray*}{V'}^{(1)}&=&\alpha {V'}^{(3)}+(4\beta+\frac{1}{2}\alpha)(ad_\omega^{-1}[Q,[Q,ad_\omega^{-1}Q]_{\mathfrak k}]+\partial_x^{-1}[Q,ad_\omega^{-1}[Q,[Q,ad_\omega^{-1}Q]_{\mathfrak k}]]_{\mathfrak k})\\
    &\to&
    \alpha V^{(3)}-(4\beta
    +\frac{1}{2}\alpha)[\sigma_3,ad_{P}^3\sigma_3]={V}^{(1)};\end{eqnarray*}
    \item[] $\lambda^0$: \begin{eqnarray*}{V'}^{(0)}&=&\alpha {V'}^{(4)}+(4\beta+\frac{1}{2}\alpha)(M'_{\mathfrak m}+M'_{\mathfrak k})
    ~\to~ \alpha V^{(4)}+(4\beta+\frac{1}{2}\alpha)(M_{\mathfrak m}+M_{\mathfrak k})={V}^{(0)}\end{eqnarray*}
\end{itemize}
Here in the last limit, we have used the facts that
    \begin{eqnarray*}
    M'_{\mathfrak m}&=&ad_\omega^{-2}[Q,[Q,ad_\omega^{-1}Q]_{\mathfrak k}]_x-ad_\omega^{-1}[Q,ad_\omega^{-1}[Q,[Q,ad_\omega^{-1}Q]_{\mathfrak k}]]_{\mathfrak m}\\
    &&-ad_\omega^{-1}[Q,\partial_x^{-1}[Q,ad_\omega^{-1}[Q,[Q,ad_\omega^{-1}Q]_{\mathfrak k}]]_{\mathfrak k}]\\
    &\to &[\sigma_3,(4P^3)_x]=M_{\mathfrak m};\\
    M'_{\mathfrak k}&=&\partial_x^{-1}[Q,M'_{\mathfrak m}]_{\mathfrak k}\to\partial_x^{-1}[P,M_{\mathfrak m}]=M_{\mathfrak k}.
    \end{eqnarray*}
From the above exploitations, we have that when $\Delta \mu\to0$,
$$
{\widetilde A}^{\prime}\to {\widetilde A},$$
where ${\widetilde A}$ given by (\ref{tA-con}).
Furthermore, when $\Delta \mu\to0$,
\begin{eqnarray*}
{\widetilde K}'&=&-\lambda^2(4\beta+\frac{1}{2}\alpha)[Q,[Q,ad_{\omega}^{-1}Q]_{\mathfrak k}]~dx\wedge dt\\
&\to&-\lambda^2(4\beta+\frac{1}{2}\alpha)[P,[P,-[\sigma_3,P]]]~dx\wedge dt=-\lambda^2(4\beta+\frac{1}{2}\alpha)ad_P^3\sigma_3~dx\wedge dt={\widetilde K},
\end{eqnarray*}
where ${\widetilde K}$ is given by (\ref{gauge5}).
This indicates that the limits of the connection 1-form ${\widetilde A}^{\prime}$ given by (\ref{Newone}) and the curvature 2-form ${\widetilde K}^{\prime}$ given by (\ref{Newtwo}) are respectively the connection 1-form ${\widetilde A}$ given by (\ref{tA-con}) and the curvature 2-from ${\widetilde K}$ given by (\ref{gauge5}). Hence, as PDEs with prescribed curvature representations, the third-order model (\ref{Q4s}) reverts to the third-order non-integrable model (\ref{gauge}) when $\Delta \mu\to0$. Lemma \ref{Le2} is proved completely.
\end{proof}

Combining Lemma \ref{Le1} with Lemma \ref{Le2}, we thus have

\begin{Theorem}\label{Th3}
When $F_s^{\mu}(k_1,\cdots,k_s)$ equi-collapses to $G(k,n-k)$, the three vortex models (\ref{GGLIE}), (\ref{GGKDV}) and (\ref{Ga4}) in the Hermitian reductive Lie algebra ${\mathfrak g}=\mathfrak u(n)$ revert respectively to
the vortex models (\ref{GLIE1}), (\ref{GKDV1}) and (\ref{GFM11}) in the Hermitian symmetric Lie algebra ${\mathfrak h}=\mathfrak u(n)$.
\end{Theorem}

Theorem \ref{Th3} indicates that the three basic models (\ref{GGLIE}), (\ref{GGKDV}) and (\ref{Ga4}) are respectively the leading-order, second-order and third-order vortex models in the Hermitian reductive non-symmetric Lie algebra ${\mathfrak g}=u(n)$. This establishes the basic models of vortex filaments in ${\mathfrak g}$, up to the third-order approximation.

\smallskip
Now we shall set $\mathbb CP^2=\frac{U(3)}{U(1)\times U(2)}\subset F^{\mu}_3(1,1,1)=\frac{U(3)}{U(1)\times U(1)\times U(1)}$ and the reduction of their related leading-order vortex models as a concrete example to demonstrate the reduction of ${\rm F}_s^{\rm \mu}(k_1,\cdots,k_s)$ to $G(k,n-k)$ and their vortex models. One knows that $F^{\mu}_3(1,1,1)$ possesses a canonically defined connection with non-vanishing torsion (see, for example, \cite{FK,Nomizu}). Evaluated at the fixed point $o$, the curvature and torsion of this connection are formulated in terms of the Lie bracket on $\mathfrak m$: for $X,Y,Z\in \mathfrak m$,
$$
(R(X,Y)Z)_{o}=-[[X,Y]_{\mathfrak k},Z],~~~T(X,Y)_{o}=-[X,Y]_{\mathfrak m}.$$
Denoting by $e_{ij}$ the $n\times n$- matrix whose entry in the $i$-th row and $j$-th column is $1$ and 0 otherwise, we have that, for $Q\in \mathfrak m$, $Q=\sum\limits_{1\le i<j\le n}(\varphi_{ij}e_{ij}-\bar\varphi_{ji}e_{ji})$.
The generalized nonlinear Schr\"odinger equation (i.e., the $N$-wave equation called in \cite{FK}) in $\mathfrak g=\mathfrak u(n)$ can be expressed as follows (refer to \cite{FK}):
for $\alpha\in \{(ij)~|~1\le i<j\le n\}$,
\begin{eqnarray}
{\varphi_{\alpha}}_t
&=& \frac{1}{\alpha(\omega)}{\varphi_{\alpha}}_{xx}
+\sum_{(\sigma-\delta)(\omega)=0}\frac{1}{\sigma(\omega)}R_{\sigma,-\delta,\beta}^\alpha \varphi_{\beta}\varphi_{\sigma}\overline{\varphi}_{\delta}\nonumber\\
&&-\sum_{\substack{\beta,\sigma\in\Theta^+\\\alpha=\beta-\sigma\in\Theta^+}}
\left( \frac{1}{\sigma(\omega)}\varphi_{\beta}{\overline{\varphi}_{\sigma}}_x
+\frac{1}{\beta(\omega)}{\varphi_{\beta}}_x\overline{\varphi}_{\sigma} \right) T_{\beta,-\sigma}^{\alpha}\nonumber\\
&&+\sum_{\substack{\beta,\sigma\in\Theta^+\\\alpha=\beta+\sigma}}\frac{1}{\sigma(\omega)}\varphi_\beta {\varphi_{\sigma}}_x \, T_{\beta,\sigma}^{\alpha},
\label{1st-root-eq}
\end{eqnarray}
where $\Theta^+=\{e_{ij}\}_{1\le i<j\le n}$.
Now, in the reductive Lie algebra $\mathfrak g=\mathfrak u(3)$, we set
\begin{eqnarray*}
    e_1&=&e_{12},~~~~~~~~e_2=e_{23},~~~~~~~~e_3=e_{13},\\
    e_{-1}&=&-e_{21},~
    ~~~~~e_{-2}=-e_{32},~~~~~e_{-3}=-e_{31}
\end{eqnarray*}
with $[e_1,e_2]=e_{1+2}=e_3$, and have $\Theta^+=\{e_1,e_2,e_3\}$. It is a direct verification that
\begin{eqnarray*}
    \left[\omega,e_1\right]&=&\alpha_1(\omega)e_1=\frac{\sqrt{-1}}{2}(\mu_1-\mu_2)e_1;~~\left[\omega,e_{-1}\right]=\alpha_{-1}(\omega)e_{-1}=-\frac{\sqrt{-1}}{2}(\mu_1-\mu_2)e_{-1};\\
    \left[\omega,e_2\right]&=&\alpha_2(\omega)e_1=\frac{\sqrt{-1}}{2}(\mu_2-\mu_3)e_2;~~\left[\omega,e_{-2}\right]=\alpha_{-2}(\omega)e_{-2}=-\frac{\sqrt{-1}}{2}(\mu_2-\mu_3)e_{-2};\\
    \left[\omega,e_3\right]&=&\alpha_3(\omega)e_1=\frac{\sqrt{-1}}{2}(\mu_1-\mu_3)e_3;~~\left[\omega,e_{-3}\right]=\alpha_{-3}(\omega)e_{-3}=-\frac{\sqrt{-1}}{2}(\mu_1-\mu_3)e_{-3}.
\end{eqnarray*}
From
\begin{eqnarray*}
    R(e_\beta,e_\gamma)e_\eta&=&-[[e_\beta,e_\gamma]_{\mathfrak k},e_{\eta}]=R_{\beta,\gamma,\eta}^{\alpha}e_{\alpha},\\
    T(e_\beta,e_\gamma)&=&-[e_\beta,e_\gamma]_{\mathfrak m}=T_{\beta,\gamma}^{\alpha}e_{\alpha},
\end{eqnarray*}
we see that the associated non-zero components of the curvature and torsion tensors are
\begin{eqnarray}
    R_{1,-1,1}^1&=&R_{2,-2,2}^2=R_{3,-3,3}^3=2,\nonumber\\
    R_{1,-1,3}^3&=&R_{3,-3,1}^1=R_{3,-3,2}^2=R_{2,-2,3}^3=1,\nonumber\\
    R_{1,-1,2}^2&=&R_{2,-2,1}^1=-1, \nonumber\\
    T_{1,2}^3&=&T_{3,-1}^2=-1,\nonumber\\
    T_{3,-2}^1&=&T_{2,1}^3=1. \nonumber\label{T}
\end{eqnarray}
Substituting Eq.(\ref{1st-root-eq}), we have exactly Eq.(\ref{3-wave}). When $\Delta \mu\to 0$, or in other words, $\mu_1\to 1$, $\mu_2, \mu_3\to -1$ and $\varphi_{2} \to 0$ equi-convergently, from (\ref{3-wave}) we have
\begin{eqnarray}
\left\{\begin{array}{l}
\sqrt{-1} {\varphi_1}_t={\varphi_1}_{x x}+2\varphi_1\left(\left|\varphi_1\right|^2-\lim\limits_{\Delta \mu\to0}\frac{\left|\varphi_2\right|^2}{\mu_2-\mu_3}+\frac{1}{2}\left|\varphi_3\right|^2\right)-2\lim\limits_{\Delta \mu\to0}\frac{\varphi_3 \bar{\varphi_2}_x}{\mu_2-\mu_3}, \\
0=2\lim\limits_{\Delta \mu\to0}\frac{{\varphi_2}_{xx}}{\mu_2-\mu_3}+\varphi_3 \bar{\varphi_1}_x+{\varphi_3}_x \bar{\varphi}_1,\\
\sqrt{-1} {\varphi_3}_t={\varphi_3}_{x x}+2\varphi_3\left(\frac{1}{2}\left|\varphi_1\right|^2+\lim\limits_{\Delta \mu\to0}\frac{\left|\varphi_2\right|^2}{\mu_2-\mu_3}+\left|\varphi_3\right|^2\right)-2\lim\limits_{\Delta \mu\to0}\frac{{\varphi_2}_x \varphi_1}{\mu_2-\mu_3}.
\end{array}\label{limit-3-wave}\right.
\end{eqnarray}
Here we have used $\lim\limits_{\Delta \mu\to0} {\varphi_2}_t=0$ by Lemma 2 (a), the existence of $\lim\limits_{\Delta \mu\to0}\frac{{\varphi_2}_x}{\mu_2-\mu_3}$ by Lemma 2 (b) and so for $\lim\limits_{\Delta \mu\to0}\frac{{\varphi_2}_{x x}}{\mu_2-\mu_3}$ etc.
From the second equation in (\ref{limit-3-wave}), we see that $\lim\limits_{\Delta \mu\to0}\left(\frac{{\varphi_2}_x}{\mu_2-\mu_3}\right)_{x}=-\frac{1}{2}(\bar{\varphi}_1\varphi_3)_{x}$.  Hence $\lim\limits_{\Delta \mu\to0}\frac{{\varphi_2}_x}{\mu_2-\mu_3}=-\frac{1}{2}\bar{\varphi}_1\varphi_3$, $\lim\limits_{\Delta \mu\to0}\frac{\varphi_{2}}{\mu_2-\mu_3}=-\frac{1}{2}\int^x\bar{\varphi}_1\varphi_3$ by ignoring the integration constants and consequently $\lim\limits_{\Delta \mu\to0}\frac{\left|\varphi_2\right|^2}{\mu_2-\mu_3}=0$.  Substituting them into the other two equations in (\ref{limit-3-wave}), we have
$$\left\{\begin{array}{l}
\sqrt{-1} {\varphi_1}_t={\varphi_1}_{x x}+2\varphi_1\left(|\varphi_1|^2+|\varphi_3|^2\right), \\
\sqrt{-1} {\varphi_3}_t={\varphi_3}_{x x}+2\varphi_3\left(|\varphi_1|^2+|\varphi_3|^2\right).\\
\end{array}\right.$$
This is exactly the 2-vector nonlinear Schr\"odinger equation relating to $\mathbb {CP}^2\hookrightarrow \mathfrak u(3)$. The reduction of Eq.(\ref{3-wave}) to the  2-vector nonlinear Schr\"odinger equation is obtained, and hence the reduction of the first-order vortex model in the reductive Lie algebra $\mathfrak g=\mathfrak u(3)$ to the first-order vortex model in the symmetric Lie algebra $\mathfrak h=\mathfrak u(3)$, as the complex flag manifold $F^{\mu}_3(1,1,1)$ equi-collapses to $\mathbb CP^2$ in $\mathfrak u(3)$.

\smallskip
In summary, we have extended the three basic vortex models from the Hermitian symmetric Lie algebra $\mathfrak h=\mathfrak u(n)$ to the general Hermitian reductive Lie algebra $\mathfrak g=\mathfrak u(n)$. In this process, the integrable Fordy-Kulish's NLS and Langer-Perline's LIE hierarchies in Hermitian reductive Lie algebras are introduced. The equivalence between $(L+2)$-th Fordy-Kulish's NLS equation and $L$-th Langer-Perline's LIE equation ($L\ge1$) is proved. The fascinating mechanism hidden in the reduction of ${\rm F}_s^{\rm \mu}(k_1,\cdots,k_s)$ to $G(k,n-k)$ and reduction of their related vortex models are investigated in the framework of equi-convergence in $\mathfrak g=\mathfrak u(n)$.

There are many questions unclear in this aspect. For example, (1)  Terng and Uhlenbeck proved that the leading-order vortex model (\ref{GLIE}) in Hermitian symmetric Lie algebras is equivalent to the equation of 1-d Schr\"odinger flow on the Grassmannian $Gr(k,n-k)$. Is this geometric interpretation still true analogously for the leading-order vortex model (\ref{GGLIE}) in Hermitian reductive Lie algebras? (2) The long-time existence of smooth solutions to the leading-order, second-order and third-order vortex models in Hermitian reductive Lie algebras. (3) Except the Hermitian reductive Lie algebras of type A III, there are C I, D III and BD I-types (and also exceptional
cases) (refer to [23,26]). How about the analogous results in those Hermitian reductive Lie algebras?

\section*{Acknowledgement}

The first author is supported by the National Natural Science Foundation of China (No. 12141104). The third author is supported by the National Natural Science Foundation of China (No. 12561010), the Natural Science Foundation of Jiangxi Province (No. 20232BAB201006) and Degree \& Graduate Education Teaching Reform Research Project of Jiangxi Province (No. JXYJG-2023-169).

\appendix
\section{Appendix}
		
\subsection{Detailed computation of \texorpdfstring{$V^{'(4)}$}{V'(4)} } \label{A.1}
Recall the recurrence relations given in  (\ref{eq:1}) (\ref{eq:2}) and (\ref{eq:3}) in \S3:
		\begin{align*}
			V^{'(0)} &= \omega,  \\
			V_{\mathfrak m}^{'(j)} &= ad_{\omega}^{-1}\big(\partial_x V_{\mathfrak m}^{'(j-1)} - [Q, V_{\mathfrak k}^{'(j-1)}] - [Q,V_{\mathfrak m}^{'(j-1)}]_{\mathfrak m}\big), \quad j=1, \cdots, L,  \\
			V_{\mathfrak k}^{'(j)} &= \big(\partial_x ^{-1}[Q, V_{\mathfrak m}^{'(j)}]\big)_{\mathfrak k}, \quad j=1, \cdots, L,
		\end{align*}
we have the followings step by step (the integrate constants are chosen to be zero):
			\begin{eqnarray*}
				V_{\mathfrak m}^{'(1)} &=& ad_{\omega}^{-1}\big(\partial_x V_{\mathfrak m}^{'(0)} - [Q, V_{\mathfrak k}^{'(0)}] - [Q,V_{\mathfrak m}^{'(0)}]_{\mathfrak m}\big)\\
				&=&ad_{\omega}^{-1}(-[Q,\omega])=ad_{\omega}^{-1}[\omega,Q]=Q,\\
				V_{\mathfrak k}^{'(1)} &=& \big(\partial_x ^{-1}[Q, V_{\mathfrak m}^{'(1)}]\big)_{\mathfrak k}=\big(\partial_x ^{-1}[Q, Q]\big)_{\mathfrak k}{=}0;
			\end{eqnarray*}
			\begin{eqnarray*}
				V_{\mathfrak m}^{'(2)} &=& ad_{\omega}^{-1}\big(\partial_x V_{\mathfrak m}^{'(1)} - [Q, V_{\mathfrak k}^{'(1)}] - [Q,V_{\mathfrak m}^{'(1)}]_{\mathfrak m}\big)\\
				&=&ad_{\omega}^{-1}\big(Q_x-[Q,0]-[Q,Q]_{\mathfrak m}\big)=ad_{\omega}^{-1}Q_x,\\
				V_{\mathfrak k}^{'(2)} &=& \big(\partial_x ^{-1}[Q, V_{\mathfrak m}^{'(2)}]\big)_{\mathfrak k}=\big(\partial_x ^{-1}[Q, ad_{\omega}^{-1}Q_x]\big)_{\mathfrak k}\\
				&{=}&\frac{1}{2}[Q,ad_{\omega}^{-1}Q]_{\mathfrak k};
			\end{eqnarray*}
			\begin{eqnarray*}
				V_{\mathfrak m}^{'(3)} &=& ad_{\omega}^{-1}\big(\partial_x V_{\mathfrak m}^{'(2)} - [Q, V_{\mathfrak k}^{'(2)}] - [Q,V_{\mathfrak m}^{'(2)}]_{\mathfrak m}\big)\\
				&=&ad_\omega^{-2}Q_{xx}-\frac{1}{2}ad_{\omega}^{-1}[Q,[Q,ad_{\omega}^{-1}Q]_{\mathfrak k}]-ad_{\omega}^{-1}[Q,ad_{\omega}^{-1}Q_x]_{\mathfrak m},\\
				V_{\mathfrak k}^{'(3)} &=& \big(\partial_x ^{-1}[Q, V_{\mathfrak m}^{'(3)}]\big)_{\mathfrak k}\\
				&=&\big(\partial_x^{-1}[Q,ad_\omega^{-2}Q_{xx}-\frac{1}{2}ad_{\omega}^{-1}[Q,[Q,ad_{\omega}^{-1}Q]_{\mathfrak k}]-ad_{\omega}^{-1}[Q,ad_{\omega}^{-1}Q_x]_{\mathfrak m}]\big)_{\mathfrak k};
			\end{eqnarray*}
			\begin{eqnarray}
				V_{\mathfrak m}^{'(4)} &=& ad_{\omega}^{-1}\big(\partial_x V_{\mathfrak m}^{'(3)} - [Q, V_{\mathfrak k}^{'(3)}] - [Q,V_{\mathfrak m}^{'(3)}]_{\mathfrak m}\big)\nonumber\\
				&=&ad_{\omega}^{-1}\big(ad_\omega^{-2}Q_{xxx}-\frac{1}{2}(ad_{\omega}^{-1}[Q,[Q,ad_{\omega}^{-1}Q]_{\mathfrak k}])_x-(ad_{\omega}^{-1}[Q,ad_{\omega}^{-1}Q_x]_{\mathfrak m})_x\nonumber\\
				&&-ad_{\omega}^{-1}[Q,\big(\partial_x^{-1}[Q,ad_\omega^{-2}Q_{xx}-\frac{1}{2}ad_{\omega}^{-1}[Q,[Q,ad_{\omega}^{-1}Q]_{\mathfrak k}]-ad_{\omega}^{-1}[Q,ad_{\omega}^{-1}Q_x]_{\mathfrak m}]\big)_{\mathfrak k}]\nonumber\\
				&&-ad_{\omega}^{-1}[Q,ad_\omega^{-2}Q_{xx}-\frac{1}{2}ad_{\omega}^{-1}[Q,[Q,ad_{\omega}^{-1}Q]_{\mathfrak k}]-ad_{\omega}^{-1}[Q,ad_{\omega}^{-1}Q_x]_{\mathfrak m}]_{\mathfrak m},\nonumber\\
				V_{\mathfrak k}^{'(4)} &=& \big(\partial_x ^{-1}[Q, V_{\mathfrak m}^{'(4)}]\big)_{\mathfrak k}.\nonumber
			\end{eqnarray}
This gives the explicit expression of ${V'}^{(4)}=V_{\mathfrak m}^{'(4)}+V_{\mathfrak k}^{'(4)}$.

\subsection{ Expressions of $f$ and $g$ in Eq.(\ref{Ga4})}\label{A.2}
		Recall that the third-order vortex model (\ref{GFM12}) is
		\begin{equation}\label{third}
			\widetilde\gamma_t = \left[\widetilde\gamma,\alpha(E^{-1}V_{\mathfrak m}^{'(4)}E)+\left(4\beta + \frac{1}{2}\alpha\right)(E^{-1}M_{\mathfrak{m}}'E)\right],
\nonumber
		\end{equation}
		where
		\begin{eqnarray}
			E^{-1}V_{\mathfrak m}^{'(4)}E&=&E^{-1}\big(ad_{\omega}^{-1}((V_{\mathfrak{m}}^{'(3)})_x-[Q,V_{\mathfrak m}^{'(3)}]_{\mathfrak m}-[Q,V_{\mathfrak k}^{'(3)}])\big)E\nonumber\\
			&=&E^{-1}\big(ad_{\omega}^{-3}Q_{xxx}-\frac{1}{2}(ad_{\omega}^{-2}[Q,[Q,ad_{\omega}^{-1}Q]_{\mathfrak k}])_x-(ad_{\omega}^{-2}[Q,ad_{\omega}^{-1}Q_x]_{\mathfrak m})_x\nonumber\\
			&&-ad_{\omega}^{-1}[Q,V_{\mathfrak m}^{'(3)}]_{\mathfrak m}-ad_{\omega}^{-1}[Q,V_{\mathfrak k}^{'(3)}]\big)E,  \nonumber\label{E^{-1}V^{'(4)}E}\\
			E^{-1}M'_{\mathfrak m}E&=&E^{-1}\big(ad_\omega^{-2}\left[Q,\left[Q,ad_\omega^{-1}Q\right]_{\mathfrak k}\right]_x - ad_\omega^{-1}\left[Q,ad_\omega^{-1}\left[Q,\left[Q,ad_\omega^{-1}Q\right]_{\mathfrak k}\right]\right]_{\mathfrak m}\nonumber\\
			&&-ad_\omega^{-1}\left[Q,\partial_x^{-1}\left[Q,ad_\omega^{-1}\left[Q,\left[Q,ad_\omega^{-1}Q\right]_{\mathfrak k}\right]\right]_{\mathfrak k}\right]\big)E.\label{2222}
		\end{eqnarray}
From the definition that $\widetilde\gamma=E^{-1}\omega E$, we have
		\begin{eqnarray*}
			\widetilde\gamma_x&=&E^{-1}[\omega,Q]E,\\
			\widetilde\gamma_{xx}&=&E^{-1}\big([\omega,Q_x]+[[\omega,Q],Q]\big)E,\\
			\widetilde\gamma_{xxx}&=&E^{-1}\big([\omega,Q_{xx}]+2[[\omega,Q_x],Q]+[[\omega,Q],Q_x]+[[[\omega,Q],Q],Q]\big)E,\\
			\widetilde\gamma_{xxxx}&=&E^{-1}\big([\omega,Q_{xxx}]+3[[\omega,Q_{xx}],Q]+3[[\omega,Q_x],Q_x]+3[[[\omega,Q_x],Q],Q]\\
			&&+2[[[\omega,Q],Q_x],Q]+[[\omega,Q],Q_{xx}]+[[[\omega,Q],Q],Q_x]+[[[[\omega,Q],Q],Q],Q]\big)E.
		\end{eqnarray*}
		Then we obtain
		\begin{eqnarray*}
			E^{-1}[\omega,Q_{xxx}]E&=&\widetilde\gamma_{xxxx}-E^{-1}\big(3[[\omega,Q_{xx}],Q]+3[[\omega,Q_x],Q_x]+3[[[\omega,Q_x],Q],Q]\\
			&&+2[[[\omega,Q],Q_x],Q]+[[\omega,Q],Q_{xx}]+[[[\omega,Q],Q],Q_x]\\
			&&+[[[[\omega,Q],Q],Q],Q]\big)E,
		\end{eqnarray*}
		\begin{eqnarray}
			E^{-1}(ad_{\omega}^{-3}Q_{xxx})E
			&=&\sum_{1\leq i<j \leq s}\frac{16}{(\mu_i-\mu_j)^4}\big(({\widetilde\gamma}_{xxxx})_{ij}+({\widetilde\gamma}_{xxxx})_{ji}\big)-E^{-1}\big(ad_{\omega}^{-4}(3[[\omega,Q_{xx}],Q]\nonumber\\
			&&+3[[\omega,Q_x],Q_x]+3[[[\omega,Q_x],Q],Q]+2[[[\omega,Q],Q_x],Q]+[[\omega,Q],Q_{xx}]\nonumber\\
			&&+[[[\omega,Q],Q],Q_x]+[[[[\omega,Q],Q],Q],Q])\big)E.\nonumber
		\end{eqnarray}
and
		\begin{eqnarray}
			E^{-1}V_{\mathfrak m}^{'(4)}E&=&\sum_{1\leq i<j \leq s}\frac{16}{(\mu_i-\mu_j)^4}\big(({\widetilde\gamma}_{xxxx})_{ij}+({\widetilde\gamma}_{xxxx})_{ji}\big)-E^{-1}\big(ad_{\omega}^{-4}(3[[\omega,Q_{xx}],Q]\nonumber\\
			&&+3[[\omega,Q_x],Q_x]+3[[[\omega,Q_x],Q],Q]+2[[[\omega,Q],Q_x],Q]+[[\omega,Q],Q_{xx}]\nonumber\\
			&&+[[[\omega,Q],Q],Q_x]+[[[[\omega,Q],Q],Q],Q])+\frac{1}{2}(ad_{\omega}^{-2}[Q,[Q,ad_{\omega}^{-1}Q]_{\mathfrak k}])_x\nonumber\\
			&&+(ad_{\omega}^{-2}[Q,ad_{\omega}^{-1}Q_x]_{\mathfrak m})_x+[Q,V_{\mathfrak m}^{'(3)}]_{\mathfrak m}+[Q,V_{\mathfrak k}^{'(3)}]\big)E.\label{3333}
		\end{eqnarray}
Based on (\ref{2222}) and (\ref{3333}), we set
		\begin{eqnarray*}
			f&=&-E^{-1}\bigg(\big(ad_{\omega}^{-4}(3[[\omega,Q_{xx}],Q]+[[\omega,Q],Q_{xx}]+3[[\omega,Q_x],Q_x]+3[[[\omega,Q_x],Q],Q] \\
			&&+2[[[\omega,Q],Q_x],Q]+[[[\omega,Q],Q],Q_x]+[[[[\omega,Q],Q],Q],Q])\big) +\frac{1}{2}(ad_{\omega}^{-2}[Q,[Q,ad_{\omega}^{-1}Q]_{\mathfrak k}])_x\\
			&&+(ad_{\omega}^{-2}[Q,ad_{\omega}^{-1}Q_x]_{\mathfrak m})_x+ad_{\omega}^{-1}[Q,V_{\mathfrak k}^{'(3)}]+[Q,V_{\mathfrak m}^{'(3)}]_{\mathfrak m}\bigg)E\\
			g&=&E^{-1}M'_{\mathfrak m}E \\
			&=&E^{-1}\big(ad_\omega^{-2}\left[Q,\left[Q,ad_\omega^{-1}Q\right]_{\mathfrak k}\right]_x - ad_\omega^{-1}\left[Q,ad_\omega^{-1}\left[Q,\left[Q,ad_\omega^{-1}Q\right]_{\mathfrak k}\right]\right]_{\mathfrak m}\\
			&&-ad_\omega^{-1}\left[Q,\partial_x^{-1}\left[Q,ad_\omega^{-1}\left[Q,\left[Q,ad_\omega^{-1}Q\right]_{\mathfrak k}\right]\right]_{\mathfrak k}\right]\big)E.
		\end{eqnarray*}
One verifies that $E^{-1}QE=\sum\limits_{1\leq i<j\leq s}\frac{-4}{(\mu_i-\mu_j)^2}\big([{\widetilde\gamma},{\widetilde\gamma}_x]_{ij}+[{\widetilde\gamma},{\widetilde\gamma}_x]_{ji}\big)$ is expressed by $\widetilde\gamma$ and $\widetilde\gamma_x$. It is also true for $E^{-1}Q_xE$ (resp. $E^{-1}Q_{xx}E$) up to the second-order derivative (resp. third-order derivative) of $\widetilde\gamma$. Therefore, $f$ and $g$ are the lower order derivative terms of $\widetilde\gamma$.

\end{document}